\documentclass[10pt]{amsart}
\usepackage{amsthm,amsmath,amssymb}

\theoremstyle{plain}
\newtheorem{theorem}{Theorem}
\newtheorem{corollary}[theorem]{Corollary}
\newtheorem{lemma}[theorem]{Lemma}
\newtheorem{proposition}[theorem]{Proposition}
\newtheorem{conjecture}[theorem]{Conjecture}
\theoremstyle{definition}
\newtheorem{example}[theorem]{Example}
\theoremstyle{remark}
\newtheorem{remark}[theorem]{Remark}

\newcommand\CC{{\mathbf C}}
\newcommand\RR{{\mathbf R}}
\newcommand\ZZ{{\mathbf Z}}
\newcommand\NN{{\mathbf N}}
\newcommand\BB{{\mathbf B}}
\newcommand\PP{{\mathbf P}}   
\newcommand\spr[1]{\langle#1\rangle}

\newcommand\oz{\overline z}
\newcommand\ow{\overline w}

\newcommand\jedna{\mathbf1}
\newcommand\F[5]{{}_#1\!F_#2\Big(\begin{matrix}#3\\#4\end{matrix}\Big|#5\Big)}     
\newcommand\FK[5]{{}_#1{\mathcal F}_#2\Big(\begin{matrix}#3\\#4\end{matrix}\Big|#5\Big)}  
\newcommand\FKo[5]{{}_#1{\mathcal F}^\Omega_#2\Big(\begin{matrix}#3\\#4\end{matrix}\Big|#5\Big)}  

\newcommand\cO{\mathcal O}
\newcommand\cN{\mathcal N}  
\newcommand\cnm{\cN^m}
\newcommand\cnmv{\cnm_\nu}
\newcommand\cP{\mathcal P}  
\newcommand\cpm{\cP^m}
\newcommand\cpmv{\cpm_\nu}
\newcommand\cR{\mathcal R}  
\newcommand\crm{\cR^m}
\newcommand\crmv{\crm_\nu}
\newcommand\cS{\mathcal S}  
\newcommand\csm{\cS^m}
\newcommand\csmv{\csm_\nu}
\newcommand\cA{\mathcal A}  
\newcommand\amn{\cA^\m_\nu}

\newcommand\ok{\overline k} 
\newcommand\ol{\overline l}
\newcommand\dbar{\overline\partial}
\newcommand\oD{\overline D} 
\newcommand\m{{\mathbf m}}  
\newcommand\n{{\mathbf n}}
\newcommand\auto{\operatorname{Aut}(\Omega)}
\newcommand\gom{\Gamma_\Omega}
\newcommand\qo{q_\Omega}   
\newcommand\rpr{\RR_+^r}
\newcommand\tg{\widetilde G}   
\newcommand\tk{\widetilde K}
\newcommand\unu{U^{(\nu)}}   
\renewcommand\a{\mathfrak a}   
\newcommand\ac{\mathfrak a^{*\CC}}
\newcommand\g{\mathfrak g}
\newcommand\phin[1]{\phi_{#1,\nu}}  
\newcommand\bk{\mathbf k}    
\newcommand\HH{\mathcal H}   
\newcommand\UU{\mathbf U}
\newcommand\GG{\mathbf G}

\newcommand\hg{\hat G}    
\renewcommand\hom{\hat\Omega}
\newcommand\hnmv{\hat{\cN}^m_\nu}
\newcommand\hsmv{\hat{\cS}^m_\nu}
\newcommand\hN{\hat N}
\newcommand\hS{\hat S}
\newcommand\hmu{\hat\mu}
\newcommand\hrho{\hat\rho}
\newcommand\hphi{\hat\phi}

\renewcommand\[{\begin{equation}}
\renewcommand\]{\end{equation}}

\allowdisplaybreaks[4]

\begin{document}

\title[Nearly holomorphic kernels]{Weighted Bergman kernels for nearly holomorphic functions on bounded symmetric domains}
\author[M.~Engli\v s]{Miroslav Engli\v s}
\address{Mathematics Institute, Silesian University in Opava,
 Na~Rybn\'\i\v cku~1, 74601~Opava, Czech Republic {\rm and}
 Mathematics Institute, \v Zitn\' a 25, 11567~Prague~1,
 Czech Republic}
\email{englis{@}math.cas.cz}
\author[E.-H.~Youssfi]{El-Hassan Youssfi}
\address{Aix-Marseille Universit\'e, I2M UMR CNRS~7373,
 39 Rue F-Juliot-Curie, 13453~Marseille Cedex 13, France}
\email{el-hassan.youssfi{@}univ-amu.fr}
\author[G.~Zhang]{Genkai Zhang}
\address{Mathematical Sciences, Chalmers University of Technology, {\rm and}
 Mathematical Sciences, G\"oteborg University, SE-412 96 G\"oteborg, Sweden}
\email{genkai{@}chalmers.se}
\thanks{Research of M.~Engli\v{s} supported by GA\v CR grant no.~21-27941S
 and RVO funding for I\v CO~67985840.
 Research of G.~Zhang supported by Swedish Research Council (VR) grants
 no.~2018-03402 and 2022-02861.}
\subjclass{Primary 32M15; Secondary 46E22, 53C35}
\keywords{Nearly holomorphic functions, polyanalytic functions, Bergman kernel, bounded symmetric domain}
\begin{abstract}
We~identify the standard weighted Bergman kernels of spaces of nearly holomorphic functions,
in~the sense of Shimura, on~bounded symmetric domains. This also yields a description of the
analogous kernels for spaces of ``invariantly-polyanalytic'' functions --- a~generalization of
the ordinary polyanalytic functions on the ball which seems to be the most appropriate one from
the point of view of holomorphic invariance. In~both cases, the~kernels turn out to be given
by certain spherical functions, or equivalently Heckman-Opdam hypergeometric functions,
and a conjecture relating some of these to a Faraut-Koranyi hypergeometric function is
formulated based on the study of low rank situations. Finally, analogous results are
established also for compact Hermitian symmetric spaces, where explicit formulas in
terms of multivariable Jacobi polynomials are~given.
\end{abstract}

\maketitle

\section{Introduction}
Let $\Omega$ be an irreducible bounded symmetric domain in~$\CC^d$, $d\ge1$,
in~its Harish-Chandra realization, and denote by $p$ its genus and by $h(z,w)$
the associated Jordan triple determinant, which is a holomorphic polynomial in
$z$ and $\ow$ on~$\CC^d$. The~standard weighted Bergman spaces on $\Omega$
are the spaces
\[ A_\nu(\Omega)\equiv A_\nu := L^2(\Omega, d\mu_\nu) \cap \cO(\Omega)  \label{TA} \]
of all holomorphic functions on $\Omega$ square-integrable with respect to the measure
\[ d\mu_\nu(z) := h(z,z)^{\nu-p} \,dz , \label{TB} \]
where $dz$ stands for the Lebesgue measure. It~is well known that $A_\nu$ is nontrivial
if and only if $\nu>p-1$, and in that case $A_\nu$ possesses a reproducing kernel ---
the \emph{weighted Bergman kernel} --- given~by
\[ K_\nu(z,w) = c_\nu h(z,w)^{-\nu}, \label{TC} \]
where $c_\nu=1/\mu_\nu(\Omega)$ is a constant which can be evaluated explicitly.

For any $\phi\in\auto$, the group of all biholomorphic self-maps of~$\Omega$,
the Jordan triple determinant satisfies the transformation rule
\[ h(\phi z,\phi w) = \frac{h(a,a)h(z,w)}{h(z,a)h(a,w)}, \qquad a=\phi^{-1}0. \label{TD} \]
(We~will mostly write just $\phi z$ instead of $\phi(z)$.)
It~follows that the Riemannian metric
\[ ds^2 = - \sum_{j,k=1}^d \frac{\partial^2\log h(z,z)}{\partial z_j\partial\oz_k} \label{TE} \]
is~invariant under~$\auto$. Recall now that, quite generally, for an arbitrary K\"ahler manifold
$\Omega$ with K\"ahler metric $ds^2=\sum_{j,k}g_{j\ok}\,dz_j\,d\oz_k$, the invariant Cauchy-Riemann
operator~$\oD$, introduced by Peetre \cite{PZ}, is~the map from functions into
holomorphic vector fields defined~by
$$ \oD f = (\oD f)^j \frac\partial{\partial z_j}, \qquad
(\oD f)^j = g^{\ol j} \dbar_l f,  $$
where we have started to employ the Einstein summation convention, and also to write for brevity
$\dbar_l:=\partial/\partial\oz_l$; namely, it~is the $\dbar$ operator combined with the Riesz
lemma identifying (0,1)-forms with holomorphic vector fields.
Here $g^{\ok j}$ is the inverse matrix to~$g_{j\ok}$.
One~can iterate this construction and set, for $m=1,2,\dots$,
$$ (\oD{}^m f)^{k_m\dots k_1} = g^{\ol_m k_m}\dbar_{l_m} \dots g^{\ol_2 k_2}\dbar_{l_2} g^{\ol_1k_1}\dbar_{l_1}f .  $$
It~turns out that the tensor field $(\oD{}^m f)^{k_m\dots k_1}$ is symmetric in the indices
$k_1,\dots,k_m$ \cite{PZ}, and in fact coincides with the contravariant derivative $f^{/k_1\dots k_m}$
with respect to the Hermitian connection~\cite{E25}. The $m$-th Cauchy-Riemann space $\cnm$ \cite{EZcr},
or the space of \emph{nearly holomorphic functions of order~$m$}, is, by~definition,
the kernel of~$\oD{}^m$:
$$ \cnm(\Omega) \equiv \cnm := \{f\in C^\infty(\Omega): \oD{}^mf=0 \text{ on }\Omega \} .  $$
An~alternative definition is due to Shimura \cite{Shi}: $\cnm$ is the vector space of all
functions on $\Omega$ that can (locally) be written as polynomials of degree $<m$ in the
derivatives~$\partial_j\Psi$, with holomorphic coefficients, where $\Psi$ is a (local)
potential for the K\"ahler metric, i.e.~$g_{j\ok}=\dbar_k\partial_j\Psi$.
(This space does not depend on the choice of the local potential~$\Psi$.)
See e.g.~Proposition~7 in \cite{EZcr} for a proof of the equivalence of these two definitions.

The~above construction applies, in~particular, to~our bounded symmetric domain $\Omega$
with the invariant metric~\eqref{TE}, possessing a global K\"ahler potential~$\Psi(z)=-\log h(z,z)$.
In~analogy with~\eqref{TA}, we~can consider the \emph{weighted Bergman spaces of nearly
holomorphic functions}
\[ \cnmv := L^2(\Omega,d\mu_\nu) \cap \cnm .  \label{TF}  \]
Of~course, if~$m=1$ then $\cN^1=\cO(\Omega)$ and $\cN^1_\nu=A_\nu$ for any~$\nu$.

For the simplest bounded symmetric domain $\Omega=\BB^d$, the unit ball of~$\CC^d$, $d\ge1$,
the Jordan triple determinant is given simply by $h(z,w)=1-\spr{z,w}$, so~that $\Psi_j=\frac{\oz_j}{1-|z|^2}$.
Nearly holomorphic functions of order $m$ on $\BB^d$ are thus precisely the polynomials
of degree $\le m-1$ in $(1-|z|^2)^{-1}\oz$, with holomorphic coefficients. In~other words,
\[ \cnm(\BB^d) = (1-|z|^2)^{1-m} \cpm(\BB^d),  \label{TH}  \]
where $\cpm(\BB^d)$ consists, by~definition, of~all linear combinations, with holomorphic coefficients,
of $(1-|z|^2)^{m-1-|\alpha|}\oz^\alpha$, where $\alpha=(\alpha_1,\dots,\alpha_d)$ is a multiindex of
length $|\alpha|:=\alpha_1+\dots+\alpha_d<m$; that~is, by~a simple check, $\cpm(\BB^d)$ consists
of all polynomials of degree $\le m-1$ in~$\oz$, with holomorphic coefficients.
(Indeed, in~one direction, $(1-|z|^2)^{m-1-|\alpha|}\oz^\alpha$ is clearly a polynomial in $\oz$
of degree $m-1$ with holomorphic coefficients; while in the other direction,
\begin{align*}
\oz^\alpha &=(|z|^2+(1-|z|^2))^{m-1-|\alpha|}\oz^\alpha \\
&= \sum_{j=0}^{m-1-|\alpha|} \binom{m-1-|\alpha|}j |z|^{2j} (1-|z|^2)^{m-1-|\alpha|-j}\oz^\alpha \\
&= \sum_{|\beta|<m-1-|\alpha|} \binom{m-1-|\alpha|}\beta z^\beta (1-|z|^2)^{m-1-|\alpha+\beta|}\oz^{\alpha+\beta}
\end{align*}
is a linear combination of $(1-|z|^2)^{m-1-|\gamma|}\oz^\gamma$ with holomorphic coefficients.)
The space $\cpm(\BB^d)$ is thus nothing else than the well-known space of \emph{$m$-analytic}
functions on the ball, as~studied by many authors. The~reproducing kernel of the space
$$ L^2(\BB^d,(1-|z|^2)^s\,dz) \cap \cP^q(\BB^d), \qquad s>-1,  $$
was recently found by the second author \cite{You} to~be
\[ P^q_{s+d+1}(z,w) := \frac{\Gamma(q+s+d)}{\pi^d\Gamma(q+s)}
 \frac{(1-\spr{w,z})^{q-1}}{(1-\spr{z,w})^{q+s+d}}
 P^{(d,s)}_{q-1}(1-2|\phi_zw|^2) ,  \label{TG}  \]
where $P^{(d,s)}_n$ denotes the Jacobi polynomial of degree $n$ with parameters~$d,s$,
and $\phi_z\in\operatorname{Aut}(\BB^d)$ is the biholomorphic self-map of $\BB^d$
interchanging $z$ and the origin. Returning to our general bounded symmetric domain~$\Omega$,
we~are thus led to define, by~analogy with~\eqref{TH}, the~space of \emph{invariantly polyanalytic
functions of order $m$} on $\Omega$~as
\[ \cpm(\Omega) \equiv \cpm := h(z,z)^{m-1} \cnm,  \label{TI}  \]
and consider the corresponding weighted Bergman spaces
\[ \cpmv := L^2(\Omega,d\mu_\nu) \cap \cpm.  \label{TJ}  \]
Our aim in this paper is to find the reproducing kernels $N^m_\nu$ and $P^m_\nu$ of the spaces
$\cnmv$ and~$\cpmv$, respectively, thus generalizing the formulas \eqref{TC}
(which corresponds to $m=1$) and~\eqref{TG} (which corresponds to $\Omega=\BB^d$).

On~an abstract level, the answer is given by the group representation theory, more specifically,
by~the Plancherel formula for certain representations of the identity connected component $G$
of the automorphism group $\auto$ of~$\Omega$. Namely, from the fact that $\oD^m f$ is a tensor,
it~follows that the action of $G$ by composition preserves the space~$\cnm$; in~other words,
$$ f\longmapsto f\circ\phi^{-1}, \qquad f\in\cnm, \;\phi\in G,  $$
is~a~representation of $G$ on~$\cnm$. In~combination with the transformation rule for~$\mu_\nu$,
\[ d\mu_\nu(\phi z) = \Big|\frac{h(a,a)^{\nu/2}}{h(z,a)^\nu}\Big|^2 \,d\mu_\nu(z),
 \qquad a=\phi^{-1}0, \;\phi\in G,  \label{TK}  \]
which follows from~\eqref{TD}, this implies that
\[ f \longmapsto \frac{h(a,a)^{\nu/2}}{h(z,a)^\nu} \, f\circ\phi^{-1}, \qquad a=\phi0, \; \phi\in G, \label{TL} \]
is~a~projective unitary representation of $G$ on~$\cnmv$. It~is now a result of the third author
\cite{GZkyoto} that for each $m=1,2,\dots$, $\cnmv$~comes as an orthogonal direct sum of irreducible
components which can be identified with certain so-called relative discrete series representations of~$G$.
Finally, a~general Plancherel formula of Shimeno \cite{Shm}, applied to these representations,
implies that the reproducing kernel at the origin of each of these irreducible components
must be a constant multiple of~$\phi_{\lambda,\ell}$, the spherical function of $G$ with parameter~$\ell$
(describing the representation, actually $\ell=\nu$) and weight~$\lambda$ (uniquely associated to each of
the irreducible components). In~this way, the reproducing kernel $N^m_\nu$ is thus expressed as a finite
sum of terms involving spherical functions. (For~the particular case of $\Omega=\BB^d$, this expression
was obtained in~\cite{GZstud}.)

Recoursing to the available theory of multivariable special functions (see e.g. Anker~\cite{Ank}),
the spherical functions $\phi_{\lambda,\ell}$ can also be expressed as Heckman-Opdam hypergeometric
functions, or, if~one wishes, as~multivariable Jacobi polynomials of Debiard \cite{DebSh3}
(and~many other authors). For~instance, the result for $\BB^d$ from \cite{GZstud} just mentioned reads
$$ N^m_{s+d+1} (z,w) = k^m_{s+d+1}(|\phi_w z|^2) $$
with
\[ k^m_{s+d+1}(t) := \sum_{l=0}^{m-1} c_l(s) \F21{-l,l-s-1}d{\frac t{t-1}}, \label{TO}  \]
where
$$ c_l(s) = \frac{(s-2l+1)\Gamma(s+d+1-l)d}{\pi^d l!\Gamma(s-l+2)},  $$
and ${}_2\!F_1$ is the ordinary (Gauss) hypergeometric function. On~the other hand, from \eqref{TG}
one can express $P^m_{s+d+1}$ and $N^m_{s+d+1}$ in terms of a single Jacobi polynomial $P^{(d,s)}_{m-1}$.
Comparing both expressions leads (after working out the details) to~the equality
\[ (1-t)^{q-1} \sum_{l=0}^{q-1} c_l(s+2q-2) \F21{-l,l-s-2q+1}d{\frac t{t-1}}
 = \frac{\Gamma(q+s+d)}{\pi^d\Gamma(q+s)} P^{(d,s)}_{q-1}(1-2t). \label{TM}  \]
It~is amusing to prove this (valid) formula directly (cf.~Lemma~\ref{lem-Y} below); note that
\[ P^{(d,s)}_n(1-2t) = \binom{n+d}d \F21{-n,n+1+s+d}{d+1}t .  \label{TN}  \]
Performing explicit computer calculations for rank 2 and rank 3 bounded symmetric domains indicates that,
analogously to the rank 1 situation just described, even for general bounded symmetric domains
the kernels $N^m_\nu$ and $P^m_\nu$ can in some cases be expressed not only as a finite sum of terms involving
Heckman-Opdam hypergeometric functions, but actually as a constant multiple of a single special function,
namely a hypergeometric function of Faraut and Koranyi \cite{FK88} with certain parameters.
We~offer a conjecture to this effect, together with some consequences that would follow;
the~latter include relations among the two kinds of hypergeometric functions,
as~well as a generalization of a theorem of Helgason \cite[Theorem~V.4.5]{He} describing,
in~effect, the~reproducing kernel for a certain space of radial functions on the complex
projective space~$\CC P^d$.

The paper is organized as follows. Section~2 contains the necessary background material on bounded symmetric domains.
Section~3 lists some elementary facts about the kernels $N^m_\nu$ and $P^m_\nu$ and discusses radial
nearly-holomorphic functions, which are relevant for the sequel.
The~expressions for $N^m_\nu$ in terms of spherical functions and Heckman-Opdam hypergeometric functions
are presented in Section~4.
Section~5 describes the computations for particular bounded symmetric domains
and the resulting conjectures mentioned above.
The~final section, Section~6, briefly treats also the dual case of compact Hermitian symmetric spaces.

\section{Prerequisites on bounded symmetric domains}
Throughout the rest of this paper, $\Omega$~will be an irreducible
bounded symmetric domain in~$\CC^d$ in its Harish-Chandra realization
(i.e.~a~Cartan domain). We~denote by $G$ the identity connected component of the group
$\auto$ of all biholomorphic self-maps of~$\Omega$, and by $K$ the stabilizer in $G$ of
the origin $0\in\Omega$. Then $K$ consists precisely of the unitary maps
on $\CC^d$ that preserve~$\Omega$, and $\Omega$ is isomorphic to the coset
space~$G/K$. We~further denote by $r,a,b$ and $p$ the rank, the characteristic
multiplicities and the genus of~$\Omega$, respectively, so~that
\[ p=(r-1)a+b+2,\qquad d=\frac{r(r-1)}2a+rb+r.  \label{UA}  \]
If~$b=0$, $\Omega$ is said to be of \emph{tube type}.

Irreducible bounded symmetric domains were completely classified by E.~Cartan.
There are four infinite series of such domains plus two exceptional domains in
$\CC^{16}$ and $\CC^{27}$. For~future reference, we~include a table with brief
descriptions of these domains and with the corresponding values of $r,a,b,p$ and~$d$.
The~symbol $\mathbf O$ stands for the division algebra of octonions.

\vbox{
\bigskip
{\offinterlineskip
\halign{\ \ #\ \hss\strut&\ #\hss&\;\hss#\cr
Domain&Description&\cr
\noalign{\hrule}
\noalign{\vskip4pt}
$I_{mn}$ & $Z\in\CC^{m\times n}$: $\|Z\|_{\CC^n\to\CC^m}<1$ & $n\ge m\ge1$\cr
         & $r=m$, $a=2$, $b=n-m$, $p=n+m$, $d=mn$\cr
\noalign{\medskip}
$II_n$   & $Z\in I_{nn}$, $Z=Z^t$ & $n\ge2$ \cr
         & $r=n$, $a=1$, $b=0$, $p=n+1$, $d=\frac12 n(n+1)$\cr
\noalign{\medskip}
$III_m$  & $Z\in I_{mm}$, $Z=-Z^t$ & $m\ge5$ \cr
         & $r=[\frac m2]$, $a=4$, $b=2(m-2r)$, $p=2m-2$, $d=\frac12m(m-1)$ \cr
\noalign{\medskip}
$IV_n$ & $Z\in\CC^{n\times1}$, $|Z^tZ|<1$, $1+|Z^t Z|^2-2Z^* Z>0$ & $n\ge5$ \cr
       & $r=2$, $a=n-2$, $b=0$, $p=d=n$\cr
\noalign{\medskip}
$V$ & $Z\in\mathbf O^{1\times2}$, $\|Z\|<1$&\cr
    & $r=2$, $a=6$, $b=4$, $p=12$, $d=16$\cr
\noalign{\medskip}
$V\!I$ & $Z\in\mathbf O^{3\times3}$, $Z=Z^*$, $\|Z\|<1$&\cr
       & $r=3$, $a=8$, $b=0$, $p=18$, $d=27$\cr
\noalign{\smallskip}
\noalign{\hrule}}}
\bigskip
}

The unit balls $\BB^d=I_{1d}$ are the only bounded symmetric domains of~rank~1,
and the only bounded symmetric domain with smooth boundary.

For $x\in\Omega$, $\phi_x$~will denote the (unique) geodesic symmetry which
interchanges $x$ and the origin,~i.e.
\[ \phi_x\circ\phi_x=\text{id}, \; \phi_x(0)=x, \; \phi_x(x)=0, \label{tQD} \]
and $\phi_x$ has only an isolated fixed-point. (In~fact, $\phi_x$ has only one
fixed point, namely the geodesic mid-point between $0$ and~$x$.)
Note that from the definition of $K$ it is
immediate that any $\phi\in G$ is of the form $\phi=\phi_x k$,
where $k\in K$ and $x\in\Omega$. (In~fact~$x=\phi(0).$)

It~is known that the ambient space $\CC^d=:Z$ possesses a structure of
\emph{Jordan-Banach $*$-triple system} (or~\emph{JB*-triple} for short)
for which $\Omega$ is the open unit ball. That~is, there exists a Jordan
triple product
$$ \{\cdot,\cdot,\cdot\}: Z\times Z\times Z \to Z,
\qquad x,y,z\mapsto\{x,y,z\},   $$
(linear and symmetric in $x,z$ and anti-linear in~$y$) such that
$$ \Omega = \{ z\in Z: \|\{z,z,\cdot\}\|<1 \}.    $$
Moreover, if~one uses the notation, for $x,y\in Z$,
\begin{align*}
D(x,y) &= \{x,y,\cdot\} : \; Z\to Z, \\
Q(x)   &= \{x,\cdot,x\} : \; Z\to Z,
\end{align*}
then for every $x\in\Omega$, $D(x,x)$ is Hermitian and has nonnegative
spectrum, and $iD(x,x)$ is a triple derivation. The linear operator
\[  B(x,y) = I - 2 D(x,y) + Q(x) Q(y)  \label{tBO}  \]
on $Z$ is called the \emph{Bergman operator}.

Two~vectors $x,y\in Z$ are said to be \emph{orthogonal} (in~the
Jordan-theoretic sense) if~$D(x,y)=0$, and a vector $v\in Z$ is called
a~\emph{tripotent} if $\{v,v,v\}=v$.
For~any tripotent~$v$, the~ambient space admits the \emph{Peirce decomposition}
\[ Z = Z_0(v) \oplus Z_{1/2}(v) \oplus Z_1(v)  \label{tE} \]
into the orthogonal components
$$ Z_{j/2}(v) := \{z\in Z: D(v,v)z=\frac j2 z \}.  $$
(The~orthogonality is only with respect to the inner product in~$\CC^d$,
not~in the triple-product (Jordan-theoretic) sense.)
Each~$Z_{j/2}(v)$ is a subtriple of~$Z$, and $Z_1(v)$ is a JB*-algebra under the
product $x\circ y=\{xvy\}$, with unit $v$ and involution $z^*=\{vzv\}$.
A~tripotent $v$ is called \emph{minimal} if $\dim Z_1(v)=1$.
Any~maximal set $e_1,\dots,e_r$ of pairwise orthogonal minimal tripotents is called a \emph{Jordan frame};
its~cardinality $r$ is independent of the frame and equal to the rank $r$ of~$\Omega$.
For~any Jordan frame $e_1,\dots,e_r$, we~similarly as above have the \emph{joint Peirce decomposition}
\[ Z = \bigoplus_{0\le i\le j\le r} Z_{ij}  \label{tQC}  \]
with
\[ Z_{ij} = \{z\in Z: D(e_k,e_k)z=\frac{\delta_{ik}+\delta_{jk}}2
\; \forall k=1,\dots,r \}.  \label{tQB}  \]

Given any Jordan frame $e_1,\dots,e_r$ --- which we choose and fix once and
for all from now~on --- any~$z\in Z$ has a \emph{polar decomposition}
\[ z = k (t_1 e_1 + \dots + t_r e_r)    \label{tQA}  \]
with $k\in K$ and $t_1\ge t_2\ge\dots\ge t_r\ge0$; the numbers $t_1,\dots,t_r$,
called the \emph{singular numbers}~of~$z$, are determined uniquely, but $k$
need not~be (it~is if all the $t_j$ are distinct). Further, $z\in\Omega$ if
and only if $t_1<1$, $z\in\partial\Omega$ if and only if $t_1=1$, and $z$
belongs to the Shilov boundary $\partial_e\Omega$ of $\Omega$ if and only if
$t_1=\dots=t_r=1$; that~is, if~and only if $z=ke$, where $e=e_1+\dots+e_r$
is~a~\emph{maximal tripotent}.

Since the Jordan triple product is invariant under~$K$ (i.e.~$\{kx,ky,kx\}=
k\{x,y,z\}$ $\forall k\in K$), it~is immediate from (\ref{tQB}) that under the
decomposition~(\ref{tQC}), the Bergman operator $B(z,z)$ with $z$ as in
(\ref{tQA}) is given~by
\[ B(z,z) |_{Z_{ij}} = (1-t_i^2) (1-t_j^2) I |_{Z_{ij}}  \label{tBZ}  \]
(where $t_0:=0$).

There exists a unique polynomial $h(x,y)$ on $\CC^d\times\CC^d$, holomorphic
in $x$ and anti-holomorphic in~$y$, which is $K$-invariant, in~the sense that
$$ h(kx,ky) = h(x,y) \qquad \forall k\in K,  $$
and satisfies
$$ h(z,z) = \prod_{j=1}^r (1-t_j^2) \quad\text{ for $z$ as in (\ref{tQA})}. $$
It~is known that $h(x,y)$ is irreducible, of~degree $r$ in $x$ as well as
in~$\overline y$, and $h(x,0)=h(0,x)=1$ $\forall x\in\CC^d$; also, $h(x,y)^p=\det
B(x,y)$. Further, the measure
\[  h(z,z)^{\nu-p} \, dz  \label{tQG}   \]
is~finite if and only if $\nu>p-1$, and the corresponding weighted Bergman
kernel --- i.e.~the reproducing kernel of the space of all holomorphic
functions on $\Omega$ square-integrable with respect to~(\ref{tQG}) ---
is~equal~to
\[ K_\nu(x,y) = c_\nu h(x,y)^{-\nu}   \label{tQH}  \]
where
\[ c_\nu = \frac{\gom(\nu)}{\pi^d\gom(\nu-\frac dr)} .  \label{tCN} \]
Here $\gom$ is the \emph{Gindikin-Koecher Gamma function}
$$ \gom(\nu) := \prod_{j=1}^r \Gamma\Big(\nu-\frac{j-1}2a\Big).   $$

In~the polar coordinates~(\ref{tQA}), the measures (\ref{tQG}) assume the form
\[ \int_\Omega f(z) \, h(z,z)^\nu \, dz =
 \quad c_\Omega \int_{[0,1]^r} \int_K f(k\sum_{j=1}^r \sqrt{t_j} e_j) \, dk \,d\mu_{b,\nu,a}(t), \label{tME}  \]
where $d\mu_{b,\nu,a}$ is the \emph{Selberg measure}
\[ d\mu_{b,\nu,a}(t) := \prod_{j=1}^r (1-t_j)^{\nu-p} \prod_{j=1}^r t_j^b
 \prod_{1\le i<j\le r} |t_i-t_j|^a \, dt, \label{MUS}  \]
where $dt\equiv dt_1 \, \dots \, dt_r$. Here $dk$~is the normalized Haar
measure on the (compact) group~$K$, and
\[ c_\Omega = \frac{\pi^d\Gamma(\frac a2+1)^r}{\gom(\frac{ra}2+1)\gom(\frac dr)}. \label{tCO} \]


Let $\PP$ denote the vector space of all (holomorphic) polynomials on~$\CC^d$.
We~endow $\PP$ with the \emph{Fock} (or~\emph{Fischer}) inner product
\[ \begin{aligned}
\spr{f,g}_F :&= \pi^{-d} \int_{\CC^d} f(z) \, \overline{g(z)}
\, e^{-|z|^2} \, dz \\
&= (f(\partial) g^*)(0) = (g^*(\partial)f)(0),  \vphantom{\int}
\end{aligned} \label{tFO}  \]
where
$$ g^* (z) := \overline{g(\oz)}.  $$
This makes $\PP$ into a pre-Hilbert space, and the action
$$ f\mapsto f\circ k^{-1}, \qquad k\in K,  $$
is a unitary representation of $K$ on~$\PP$. It~is a deep result of W.~Schmid
\cite{Sch} that this representation has a multiplicity-free decomposition into
irreducibles
$$ \PP = \bigoplus_\m \; \PP_\m  $$
where $\m$ ranges over all \underbar{signatures}, i.e.~$r$-tuples
$\m=(m_1,m_2,\dots,m_r)\in\ZZ^r$ satisfying $m_1\ge m_2 \ge \dots \ge m_r\ge0$.
Polynomials in $\PP_\m$ are homogeneous of degree $|\m|:=m_1+m_2+\dots+m_r$;
in~particular, $\PP_{(0)}$ are the constants and $\PP_{(1)}$ the linear
polynomials. Any holomorphic function on $\Omega$ thus has a decomposition
$f=\sum_\m f_\m$, $f_\m\in\PP_\m$, which refines the usual homogeneous expansion.

Since the spaces $\PP_\m$ are finite dimensional, they automatically
possess a reproducing kernel: there exist polynomials $K_\m(x,y)$ on
$\CC^d\times\CC^d$, holomorphic in $x$ and~$\overline y$, such that for each
$f\in\PP_\m$ and $y\in\CC^d$,
\[ f(y)=\spr{f,K_\m(\cdot,y)}_F.  \label{4}  \]
From the definition of the spaces $\PP_\m$ it also follows that the kernels
$K_\m(x,y)$ are $K$-invariant.

It~is a consequence of Schur's lemma from representation theory that for any
$K$-invariant inner product $\spr{\cdot,\cdot}$ on~$\PP$, $\PP_\m$~and $\PP_\n$
are orthogonal if $\m\neq\n$, while on each~$\PP_\m$, $\spr{\cdot,\cdot}$ is
proportional to~$\spr{\cdot,\cdot}_F$. In~particular, for the inner product
$$ \spr{f,g}_\nu := c_\nu \int_\Omega f(z) \, \overline{g(z)} \,d\mu_\nu(z)
\qquad (\nu>p-1),  $$
(with $c_\nu$ as in~(\ref{tQH})) we have, for any $f_\m\in\PP_\m$ and
$g_\n\in\PP_\n$,
\[ \spr{f_\m,g_\n}_\nu = \frac{\spr{f_\m,g_\n}_F}{(\nu)_\m}  \label{tAU} \]
(cf.~\cite{FK88}), where $(\nu)_\m$ is the \emph{generalized Pochhammer symbol}
$$ (\nu)_\m := (\nu)_{m_1} (\nu-\tfrac a2)_{m_2}
\dots (\nu-\tfrac{r-1}2a)_{m_r};  $$
here
$$ (\nu)_k := \nu (\nu+1)\dots(\nu+k-1)
\qquad \Big(=\frac{\Gamma(\nu+k)}{\Gamma(\nu)} \text{ if }
\nu\neq 0,-1,-2,\dots, \Big)  $$
is~the ordinary Pochhammer symbol.

A~consequence of the relation (\ref{tAU}) is the \emph{Faraut-Koranyi formula}
\[ \postdisplaypenalty1000000
h(x,y)^{-\nu} = \sum_\m \; (\nu)_\m K_\m(x,y)  \label{tFK}  \]
relating the reproducing kernels $K_\nu$ from (\ref{tQH}) and $K_\m$
from~(\ref{4}).


As~already mentioned, the~point $e=e_1+\dots+e_r$ belongs to the Shilov boundary $\partial_e\Omega$
of~$\Omega$. The group $K$ acts transitively on $\partial_e\Omega$, so that
$\partial_e\Omega=\{ke,k\in K\}\simeq K/L$, where $L$ is the stabilizer of $e$ in~$K$.
Each Peter-Weyl space $\PP_\m$ contains a unique $L$-invariant polynomial $\phi_\m$
satisfying the normalization condition $\phi_\m(e)=1$.
We~will sometimes write just $\phi_\m(t_1,\dots,t_r)$ instead of $\phi_\m(t_1 e_1+\dots+t_r e_r)$.
These \emph{spherical polynomials} $\phi_\m$ satisfy $\phi_{(0)}\equiv1$,
\[ \phi_{(m_1+1,m_2+1,\dots,m_r+1)}(t_1,\dots,t_r)
 =t_1\cdots t_r\,\phi_\m(t_1,\dots,t_r),   \label{XK} \]
and are related to the reproducing kernels $K_\m$ by the formula
\[  K_\m(x,e) = \frac{d_\m}{(d/r)_\m} \phi_\m(x) , \label{tag:XC}  \]
where $d_\m:=\dim\PP_\m$. It~is known that the last dimension is given by the
formula (\cite{Up}, Lemmas 2.5 and~2.6)
$$ d_\m = \frac{(d/r)_\m}{(\qo)_\m} \, \pi_\m  $$
where
\[ \qo := \frac{r-1}2 a+1  \label{tag:QO} \]
and
\[  \pi_\m := \prod_{1\le i<j\le r} \frac{m_i-m_j+\frac{j-i}2 a}
{\frac{j-i}2 a} \, \frac{(\frac{j-i+1}2 a)_{m_i-m_j}}
{(\frac{j-i-1}2 a+1)_{m_i-m_j}} .  \label{tag:XP}  \]
Thus we may rewrite (\ref{tag:XC})~as
\[ K_\m(x,e) = \frac{\pi_\m}{(\qo)_\m} \,\phi_\m(x).  \label{UKF}  \]
Combining the last formula with the fact that \cite[Lemma~3.2]{FK88}
$$ K_\m(\hbox{$\sum_j$} t_j e_j,\hbox{$\sum_j$} s_j e_j) = K_\m (\hbox{$\sum_j$} t_j s_j e_j,e),  $$
we~thus get
\[  K_\m(k\hbox{$\sum_j$} t_j e_j,k\hbox{$\sum_j$} t_j e_j)=\frac{\pi_\m}{(\qo)_\m} \,
\phi_\m(t_1^2,\dots,t_r^2) .   \label{tag:XD}  \]
The~polynomials $\phi_\m$ have also a combinatorial
interpretation in terms of \emph{Jack symmetric polynomials} $J^{(\lambda)}_\m$
with parameter~$\lambda$ (cf.~\cite{MD}, Section~10 of Chapter~VI):
namely,
\[  \phi_\m(t_1,\dots,t_r) = j_\m^{-1} \, J^{(2/a)}_\m(t_1,\dots,t_r), \label{tag:XF}  \]
where
\[  j_\m := J^{(2/a)}_\m(\underbrace{1,\dots,1}_r) =
\Big(\frac2a\Big)^{|\m|} \Big(\frac{ra}2\Big)_\m.  \label{tag:XH} \]
We will usually suppress the superscripts $(2/a)$ in the sequel.

Recall that in any Jordan algebra $J$ with unit $v$ and product $x\circ y$
an element $x$ is called \emph{invertible} if it has a (necessarily unique)
inverse $y=:x^{-1}$ satisfying $x\circ y=v$ and $x^2\circ y=x$. In~the special
case that the Jordan algebra arises as $J=Z_1(v)$ for a tripotent $v$ of the
JB*-triple $Z$ then invertibility of $z\in J$ is equivalent to the
invertibility of the operator $Q(z)$ on $J$ and $z^{-1}=Q(z)^{-1}Q(v)z$.
In~particular, taking the inverse is a rational map on $J$ that can be written
(see~e.g.~\cite[Chapter~4]{Up}) in exact (i.e.~reduced) form as $z^{-1}=p(z)/
N(z)$, where $p:J\to J$ is a polynomial which generalizes the matrix adjoint
and $N:J\to\CC$ is a polynomial called the \emph{determinant polynomial},
or~\emph{Koecher norm}, of~the Jordan algebra. In~particular, fixing a Jordan
frame $e_1,\dots,e_r$ of $Z$ the above applies to the Jordan algebras
$Z_1(e_1+\dots+e_j)$, $1\le j\le r$; we~denote the corresponding determinant
polynomials by $N_j$ and extend them to all of~$Z$ by defining
$N_j(z):=N_j(P^{(j)}_1(z))$, where $P^{(j)}_1$ is the canonical projection of
$Z$ onto $Z_1(e_1+\dots+e_j)$ given by the Peirce decomposition~(\ref{tE}).
For~a~signature~$\m$, the~\emph{conical polynomial} $N^\m$ associated with~$\m$~is
\[ N^\m := N_1^{m_1-m_2} N_2^{m_2-m_3} \cdots N_r^{m_r}.  \label{UN} \]
In~particular,
$$ N^\m\Big({\sum_{j=1}^r t_j e_j}\Big) = \prod_{j=1}^r t_j^{m_j}.  $$
Each polynomial space $\PP_\m$ is then spanned by $N^\m\circ k$, $k\in K$.
In~particular, the conical polynomials are related to the spherical polynomials by
$$ \phi_\m(z) = \int_L N^\m(lz) \, dl,  $$
where $dl$ stands for the normalized Haar measure on~$L$.

Standard references for the material in this section are \cite{Asurv}, \cite{Lo},
\cite{FK88}, or~\cite{Up}.

\section{Radial nearly holomorphic functions}
The~following relation between the nearly-holomorphic reproducing kernels $N^m_\nu$
and the invariantly-polyanalytic reproducing kernels $P^m_\nu$ is elementary.

\begin{proposition}\label{PA}
$P^m_\nu(z,w)=h(z,z)^{m-1}h(w,w)^{m-1} N^m_{\nu+2(m-1)}(z,w)$.
\end{proposition}

\begin{proof} By~their very definition~\eqref{TI}, the mapping
$$ T: f(z) \longmapsto h(z,z)^{m-1} f(z)   $$
is a bijection of $\cnm$ onto~$\cpm$. By~\eqref{TB}, $T$~clearly acts isometrically from
$L^2(\Omega,d\mu_\nu)$ into $L^2(\Omega,d\mu_{\nu-2(m-1)})$, for any $\nu\in\RR$.
Thus if $\{e_j(z)\}_j$ is an orthonormal basis for $\cnm_{\nu+2(m-1)}$, then
$\{h(z,z)^{m-1}e_j(z)\}_j$ will be an orthonormal basis for $\cpmv$.
Recalling the familiar formula for a reproducing kernel in terms of an orthonormal basis
\[ K(z,w) = \sum_j e_j(z) \overline{e_j(w)} ,  \label{VA}  \]
the assertion follows.
\end{proof}

We~also readily get a transformation formula for $N^m_\nu(z,w)$.

\begin{proposition}\label{PB}
For any $\phi\in\auto$,
\[ N^m_\nu(z,w) = \frac{h(a,a)^\nu}{h(z,a)^\nu h(a,w)^\nu} N^m_\nu(\phi z,\phi w), \qquad a:=\phi^{-1}0.  \label{VB} \]
In~particular,
\[ N^m_\nu(z,w) = h(z,w)^{-\nu} N^m_\nu(\phi_w z,0).  \label{VC}  \]
\end{proposition}

\begin{proof} Since $\oD{}^m f$ is a tensor and $\phi$ is just a coordinate change,
the~kernel $\cnm$ of $\oD{}^m$ is automatically invariant under the composition
$f\mapsto f\circ\phi$ with~$\phi$. As~already observed in the Introduction,
it~therefore follows from the transformation formula \eqref{TK} for the measure~$d\mu_\nu$
(which formula is in turn a consequence of the transformation rule \eqref{TD} for the
Jordan triple determinant, in~combination with the fact that the measure $d\mu_0$ is
$\auto$-invariant) that the operator \eqref{TL} acts unitarily on~$\cnmv$.
Employing again the formula~\eqref{VA}, we~thus obtain
$$ N^m_\nu(z,w) = \frac{h(a,a)^\nu}{h(z,a)^\nu h(a,w)^\nu} N^m_\nu(\phi^{-1}z,\phi^{-1}w),
 \qquad a:=\phi0.  $$
Replacing $\phi$ by $\phi^{-1}$ yields~\eqref{VB}, and taking $\phi=\phi_w$ in \eqref{VB} yields~\eqref{VC}.
\end{proof}

\begin{corollary}\label{PC}
$P^m_\nu(z,w)=h(w,z)^{m-1}h(z,w)^{1-m-\nu} P^m_\nu(\phi_w z,0)$.
\end{corollary}

\begin{proof} Combine the last two propositions. \end{proof}

We~have thus reduced the identification of both $N^m_\nu$ and $P^m_\nu$ to finding the
reproducing kernel $N^m_\nu(z,0)$ at zero. Note that by~\eqref{VB},
$$ N^m_\nu(kz,0) = N^m_\nu(z,0) \qquad \forall k\in K,  $$
where as before $K$ is the stabilizer of the origin $0\in\Omega$ in $G=\auto_0$;
in~other words, $N^m_\nu(\cdot,0)$ is a \emph{radial} function.
We~now proceed to identify the radial nearly holomorphic functions.

Recall that an element $z\in\CC^d$ is called \emph{quasi-invertible} with respect to
another element $w\in\CC^d$ if, by~definition, the Bergman operator \eqref{tBO} is
invertible on~$\CC^d$, and the \emph{quasi-inverse} $z^w$ is then defined~as
$$ z^w := B(z,w)^{-1}(z-Q(z)w).  $$
Note that $z^w$ is holomorphic in $z$ and anti-holomorphic in~$w$. Since $\det B(z,w)=h(z,w)^p$
does not vanish on $\Omega\times\Omega$, the quasi-inverse $z^w$ is, in~particular,
defined for all $z,w\in\Omega$. It~is now a result of \cite[formula (3.2) and Proposition~3.1]{GZkyoto},
that, first of~all, $\oD=B(z,z)\dbar$, and furthermore
\[ \partial\Psi = \oz^{\oz} = B(\oz,\oz)^{-1}(\oz-Q(\oz)\oz),  \label{VF}  \]
where as before $\partial\Psi$ stands for the vector of derivatives $\partial_j\Psi$ of the
K\"ahler potential $\Psi(z)=-\log h(z,z)$.

\begin{proposition}\label{PD}
Radial functions in $\cnm$ consist precisely of functions of the form
\[ p(z,z^z), \label{VD}  \]
where $p(z,w)$ is a polynomial in $z,\ow\in\CC^d$ of degree $<m$ in each argument
which is $K$-invariant in the sense that
\[ p(kz,kw) = p(z,w) \qquad \forall k\in K.  \label{VE}  \]
Consequently, the radial functions in $\cnm$ are the linear span of $K_\m(z,z^z)$, with $|\m|<m$.
\end{proposition}

\begin{proof} By~\eqref{VF} and the very definition of nearly-holomorphic functions,
any $f\in\cnm$ is of the form
$$ f(z)=p(z,z^z),   $$
with $p(z,w)$ holomorphic in $z\in\Omega$ and a polynomial of degree $<m$ in~$\ow$. 
Since elements of $K$ are Jordan triple automorphisms, we~have $kz^{kw}=k(z^w)$ for any $k\in K$, hence
$$ f(kz) = p(kz,k(z^z))  $$
with the same~$p$. Thus $f$ is radial if and only~if
$$ p(z,z^z) = p(kz,kz^{kz})  \qquad\forall z\in\Omega, \forall k\in K.  $$
The~last equality means that, for any fixed $k\in K$, the two holomorphic functions
$p(z,\ow^z)$ and $p(kz,k\ow^{kz})$ of $z,w\in\Omega$ coincide on the anti-diagonal $z=\ow$.
By~the well-known uniqueness principle \cite[Proposition~II.4.7]{BM}, they must coincide for all~$z,w$.
Since, for each fixed $z\in\Omega$, the image of $\Omega$ under the (non-constant anti-holomorphic)
map $w\mapsto\ow^z$ is a (nonempty) open set and $p$ is a polynomial in the second argument,
actually $p(z,y)=p(kz,ky)$ for all $z\in\Omega$ and $y\in\CC^d$, proving~\eqref{VE}.
Now it is well known basically from Schur's lemma \cite[Proposition~2]{E61} that the functions $p$
satisfying \eqref{VE} are spanned by $K_\m(z,w)$, as $\m$ ranges over all signatures.
As~$K_\m(z,w)$ is homogeneous of degree $|\m|$ in both $z$ and~$\ow$, the proposition follows.
\end{proof}

Thanks to the last proposition, we~can reduce the identification of $N^m_\nu(\cdot,0)$ to that of
the reproducing kernel at 0 of a certain space of symmetric polynomials on~$\RR^r$
(with~$r$, as~before, denoting the rank of~$\Omega$).
First of all, denote by $\crmv$ the subspace of all radial functions in~$\cnmv$,
and let $R^m_\nu(z,w)$ be its reproducing kernel. Then
\[ N^m_\nu(\cdot,0) = R^m_\nu(\cdot,0).  \label{VH} \]
Indeed, by the very definition of a reproducing kernel, $R^m_\nu(\cdot,0)$ is the (unique)
element of $\crmv$ which reproduces the value at 0 for all elements of~$\crmv$.
Now $N^m_\nu(\cdot,0)$ reproduces the value at 0 even for all elements of~$\cnmv$,
and belongs to~$\crmv$ (being radial). So~by uniqueness, \eqref{VH}~follows.

Secondly, the space $\crmv$ can be described explicitly as follows. For~ease of notation,
let us write for an $r$-tuple $t=(t_1,\dots,t_r)\in\rpr$,
\begin{gather*}
t^b:=\prod_{j=1}^r t_j^b, \qquad (1-t)^\nu:=\prod_{j=1}^r (1-t_j)^\nu, \qquad \sqrt t=t^{1/2}, \\
\frac t{1-t} :=\Big( \frac{t_1}{1-t_1},\dots,\frac{t_r}{1-t_r}\Big), \qquad dt:=dt_1\dots dt_r,
\end{gather*}
and so forth, and let $d\rho_{b,\nu,a}$ be the modified Selberg measure
\[ d\rho_{b,\nu,a}(t) := c_\Omega \; t^b (1+t)^{-\nu} \prod_{1\le i<j\le r}|t_i-t_j|^a \, dt. \label{VL} \]
Finally, if $e_1,\dots,e_r$ is a fixed Jordan frame, we~will write just $te$ for $t_1e_1+\dots+t_re_r$.
Let $\csm$ be the vector space of all symmetric polynomials of degree $<m$ in $r$ variables, denote
$$ \csmv := \csm \cap L^2(\rpr,d\rho_{b,\nu,a}),  $$
and let $S^m_\nu(x,y)$ be the reproducing kernel of~$\csmv$.

\begin{proposition}\label{PE}
The mapping $V$ from $\csm$ into functions on $\Omega$ given by
\[ Vf(k\sqrt t e) := f\Big(\frac t{1-t} \Big), \qquad k\in K, \; t\in[0,1]^r,  \label{VI}  \]
is a bijection from $\csm$ onto radial functions in~$\cnm$.
Furthermore, $V$ sends $\csmv$ unitarily onto~$\crmv$, and
\[ R^m_\nu(\cdot,0) = VS^m_\nu(\cdot,0).  \label{VK}  \]
\end{proposition}

\begin{proof} Let $z=k\sqrt t e$ be the polar decomposition of $z\in\Omega$. From the formula
\eqref{tBZ} for the action of $B(z,z)$ on the Peirce subspaces (and the similar formula for
the action of~$Q(z)$), one~gets
$$ z^z = k \frac{\sqrt t}{1-t} e .  $$
Hence
$$ K_\m(z,z^z) = K_\m(k\sqrt te,k\tfrac{\sqrt t}{1-t}e) = K_\m(\tfrac t{1-t}e,e).  $$
This~is, as~we have seen in Section~2, up~to a constant factor just the Jack symmetric
polynomial $J_\m(\frac t{1-t})$ in $r$ variables evaluated at $\frac t{1-t}$.
Since $J_\m$, $|\m|<m$, span all symmetric polynomials of degree $<m$, by the preceding
proposition the radial functions in $\cnm$ are precisely those of the form~$Vf$,
with $V$ as in \eqref{VI} and $f$ a symmetric polynomial of degree $<m$.
This proves the first assertion.

As~for the second, we~have by~\eqref{tME}
\[ \|Vf\|^2_{L^2(d\mu_\nu)} = c_\Omega \int_{[0,1]^r} \Big|f\Big(\frac t{1-t}\Big)\Big|^2 \, d\mu_{b,\nu,a}(t). \label{VJ} \]
Making the change of variable $t_j=\frac{x_j}{1+x_j}$, $x\in\rpr$, we~have
\begin{gather*}
\frac t{1-t}=x, \qquad dt= (1+x)^{-2} \,dx , \qquad t^b=x^b(1+x)^{-b},  \\
 (1-t)^{\nu-p}=(1+x)^{p-\nu}, \qquad t_i-t_j = \frac{x_i-x_j}{(1+x_i)(1+x_j)},
\end{gather*}
implying, by~a direct computation using~\eqref{UA}, that
\[ c_\Omega \,d\mu_{b,\nu,a}(t) = d\rho_{b,\nu,a}(x).  \label{VM}  \]
By~\eqref{VJ}, the second claim follows.

Finally, \eqref{VK} follows from the general formula \eqref{VA} (applied to $\csmv$ and $\crmv$),
together with the fact that under the above change of variables $x=\frac t{1-t}$, the point $t=0$
corresponds to $x=0$.
\end{proof}

We~summarize our findings so far as the main result of this section.

\begin{theorem}\label{PF}
The reproducing kernels $N^m_\nu$ and $P^m_\nu$ of the spaces $\cnmv$ and $\cpmv$, respectively,
are given~by
\begin{align*}
N^m_\nu(z,w) &= h(z,w)^{-\nu} N^m_\nu(\phi_z w,0), \\
P^m_\nu(z,w) &= \frac{h(z,z)^{m-1}h(w,w)^{m-1}}{h(z,w)^{\nu+2m-2}} N^m_{\nu+2m-2}(\phi_w z,0),
\end{align*}
where
$$ N^m_\nu(k\sqrt te,0) = S^m_\nu(\tfrac t{1-t},0),  $$
where $S^m_\nu$ is the reproducing kernel of the $L^2$ space of symmetric polynomials of degree $<m$
on $\rpr$ with respect to the measure~\eqref{VL}.
\end{theorem}

\begin{proof} Combine Propositions \ref{PA}, \ref{PB}, \ref{PD} and \ref{PE}, and the formula~\eqref{VH}.
\end{proof}

We~conclude this section by a simple observation concerning the nontriviality of the spaces $\cnmv$ and~$\cpmv$.

\begin{lemma} A~polynomial $P$ belongs to $L^2(\rpr,d\rho_{b,\nu,a})$ if and only if its degree
in each variable is less than $(\nu-p+1)/2$.
\end{lemma}

\begin{proof} Let $n_1$ be the degree of $P(x)$ in the variable~$x_1$; thus the leading term in
the $x_1$ variable is $p_1(x')x_1^{n_1}$, where the polynomial $p_1$ in the $r-1$ variables
$x'=(x_2,\dots,x_r)$ is not identically zero. The~zero-set of $p_1$ is therefore a variety in
$\RR^{r-1}$ of codimension at~least~1; we~can therefore choose a closed ball $Q$
(of~positive finite radius) lying wholly in $\{y\in\RR^{r-1}: y_j\neq y_k\text{ for all }j\neq k\}$
such that $|p_1|>0$ on~$Q$. Set~$R:=1+\sup\{\|y\|:y\in Q\}$.
Then if~$P\in L^2(\rpr,d\rho_{b,\nu,a})$, the~integral
$$ \int_R^\infty \int_Q |P(x_1,x')|^2 \, d\rho_{b,\nu,a}(x_1,x') $$
has to be finite. However, due to our choice of $Q$ and~$R$, the~integrand is $\asymp x_1^{2n_1}$
(uniformly in~$x'$), while the measure is $\asymp x_1^{b-\nu+(r-1)a}\,dx$ (uniformly in~$x'$).
Consequently, $x_1^{2n_1+(r-1)a+b-\nu}$ must be integrable at~infinity, implying that
$2n_1+(r-1)a+b-\nu=2n_1+p-2-\nu<-1$, or~$n_1<(\nu-p+1)/2$.
Similarly, $n_j<(\nu-p+1)/2$ for the degree $n_j$ of $P(x)$ in the variable~$x_j$, $j=1,\dots,r$.

Conversely, let $P(x)=x_1^{n_1}\dots x_r^{n_r}$ with $n_j<(\nu-p+1)/2$ for all~$j$.
Making again the change of variable $x=t/(1-t)$ shows that the $L^2$-norm of $P$
with respect to $c_\Omega^{-1} d\rho_{b,\nu,a}$ equals
$$ \int_{[0,1]^r} \prod_{j=1}^r \Big( t_j^{2n_j+b} (1-t_j)^{\nu-p-2n_j} \Big) \prod_{1\le i<j\le r}|t_i-t_j|^a \, dt. $$
The~second term in the integrand is bounded (by~1), while the first term yields just the product
of single-variable integrals
$$ \int_0^1 t_j^{2n_j+b} (1-t_j)^{\nu-p-2n_j} \, dt_j ,  $$
which are finite since $2n_j+b>-1$ and $\nu-p-2n_j>-1$.
\end{proof}

\begin{proposition}\label{PG}
\begin{itemize}
\item[(a)] $\cnmv\neq\{0\}$ if and only if $\nu>p-1$, and $\cpmv\neq\{0\}$ if and only if $\nu>p+1-2m$.
\item[(b)] $\cnmv\setminus\cN^{m-1}_\nu\neq\{0\}$ if and only if there exists a signature $\m$
with $|\m|=m-1$ and $m_1<\frac{\nu-p+1}2.$
\item[(c)] In~fact, $K_\m(z,z^z)\in\cnmv$ if and only if $|\m|<m$ and $m_1<\frac{\nu-p+1}2.$
\end{itemize}
\end{proposition}

\begin{proof} (a) If~$\cnmv\neq\{0\}$ then its reproducing kernel is not identically zero;
by~the last theorem, this is equivalent, in~turn, to $N^m_\nu(\cdot,0)\not\equiv0$ and
$S^m_\nu(\cdot,0)\not\equiv0$. Thus $\csmv$ contains a nonzero polynomial $p(x)$
(even one that does not vanish at the origin). By~the last lemma, necessarily $\nu-p+1>0$,
or $\nu>p-1$.

Conversely, if $\nu>p-1$, then $\cN^1_\nu=A_\nu$ is nontrivial (it~contains all bounded
holomorphic functions on~$\Omega$), hence so is $\cnmv\supset\cN^1_\nu$.

This settles the assertion for~$\cnmv$; the one for $\cpmv$ then follows from Proposition~\ref{PA}.

(b) By~the same argument as in the proof of part~(a), $\cnmv\setminus\cN^{m-1}_\nu\neq\{0\}$ if and
only if $\csmv$ contains a polynomial $P$ whose total degree is $m-1$ and whose degree in each variable
is $<\frac{\nu-p+1}2$. If~$x^\alpha$, with $\alpha$ a multiindex, is~any monomial in the top degree
homogeneous component of~$P$, then the nonincreasing rearrangement of $\alpha$ yields the desired
signature~$\m$.

(c) This follows in the same way as for part (b) from the fact that $V$ maps $K_\m(z,z^z)$ into
a (nonzero) constant multiple of the Jack polynomial $J_\m(x)$, and $J_\m(x)$ is equal to the
symmetrization~of (writing $x^\m:=x_1^{m_1}\dots x_r^{m_r}$)
\[ x^\m + \sum_{\n<\m} c_{\m\n} x^\n  \label{JACK} \]
where the sum is over (some) signatures $\n$ smaller than $\m$ with respect to the lexicographic order;
cf.~Mcdonald \cite[formula (10.13)]{MD}.
\end{proof}

\begin{corollary} \label{COCO}
$\crmv = \operatorname{span} \{ K_\m(z,z^z) : |\m|<m , \; m_1<\tfrac{\nu-p+1}2 \}$.

In~particular, if $q$ denotes the nonnegative integer such that
$$ q < \frac{\nu-p+1}2 \le q+1, $$
then
\begin{align}
\crmv &= \operatorname{span} \{K_\m(z,z^z): |\m|<m \} \qquad\text{if } m\le q+1, \label{CASE1} \\
\crmv &= \operatorname{span} \{K_\m(z,z^z): m_1\le q \} \qquad\text{if } m\ge r q+1. \label{CASE2}
\end{align}
\end{corollary}

This means that for $m\ge rq+1$, $\crmv$~and, hence, $\cnmv$~equals $\cN^{rq+1}_\nu$ ---
i.e.~the spaces $\cnmv$ ``stabilize'' and stop growing with~$m$ (for~fixed~$\nu$).
Likewise, $N^m_\nu(z,w)=N^{rq+1}_\nu(z,w)$ for all $m\ge rq+1$ if $2q-1<\nu-p\le 2q+1$.

\begin{remark}
We~pause to note that while, clearly,
$$ \cN^1 \subset \cN^2 \subset \cN^2 \subset \dots ,  $$
no~such inclusions hold for~$\cpm$, except when the rank $r=1$. More specifically, for $r>1$,
the function $\jedna$ (constant one) belongs to $\cP^1$, but not to any $\cpm$, $m\ge2$.
Indeed, $\jedna\in\cpm \iff h(z,z)^{1-m}\in\cnm$, by~\eqref{TI}; and by Proposition~\ref{PD},
the latter is equivalent~to
$$ (1-t)^{1-m} = \sum_{|\m|<m} c_\m \phi_\m(\tfrac t{1-t})  $$
with some coefficients~$c_\m$. Passing again to $x=\frac t{1-t}$, this translates into
\[ (1+x)^{m-1} = \sum_{|\m|<m} c_\m \phi_\m(x) .  \label{VN} \]
But~by the Faraut-Koranyi formula~\eqref{tFK}, the left-hand side equals
$$ \sum_\n (1-m)_\n \frac{\pi_\n (-1)^{|\n|}}{(\qo)_\n} \phi_\n(x).  $$
Since the $\phi_\m$ are linearly independent, \eqref{VN} can hold only~if
$$ (1-m)_\n=0 \qquad\text{whenever }|\n|\ge m.  $$
However, for $\n=(m-1,1)$ one has $(1-m)_\n=(-1)^m(m-1)!(m-1+\frac a2)$ which is nonzero.
So $\jedna\notin\cpm$ for $m\ge2$ if $r>1$.  \qed
\end{remark}

\section{Spherical functions}
Any $g\in G$ can be uniquely written in the form $g=k\phi_w$, with $w=g^{-1}0\in\Omega$, $k\in K$
and $\phi_w$ the geodesic reflection \eqref{tQD} interchanging $0$ and~$w$.
This yields the formula for the complex Jacobian
\[ J_g(z) = \det k \cdot (-1)^d \frac{h(w,w)^{p/2}}{h(z,w)^p},  \label{WA}  \]
which shows that the projective representation \eqref{TL} is actually nothing else than
$$ f \longmapsto J_{\phi^{-1}}^{\nu/p} \cdot f\circ\phi^{-1}, \qquad \phi\in G.  $$
In~order to make this not only projective but genuine representation if $\nu/p$ is not an integer,
one~needs to pass from $G$ to its universal cover~$\tg$. The~elements of $\tg$ can be thought of
as elements $g$ of $G$ together with a consistent choice of $\log J_g$. The~operators
\[ \unu_g: f \longmapsto J_{g^{-1}}^{\nu/p} \cdot f\circ g^{-1}, \qquad g\in\tg, \label{WB}  \]
then define a (honest, not only projective) unitary representation of $\tg$ on $L^2(\Omega,d\mu_\nu)$;
and~one has $\Omega=\tg/\tk$, where~$\tk$, the~preimage of $K$ under the covering map $\tg\to G$,
is~the universal cover of $K$ and the stabilizer of $0\in\Omega$ in~$\tg$.
(Actually one has $\tk\cong K\times\RR$, but we will not need this fact.)

Using~\eqref{WB}, one~can identify $L^2(\Omega,d\mu_\nu)$ with a subspace of $L^2(\tg/Z(\tg))$,
the $L^2$ space on the quotient of $\tg$ modulo its center $Z(\tg)$ with respect to a suitably
normalized Haar measure on~$\tg$.
Namely, for $f\in L^2(\Omega,d\mu_\nu)$, the~function $f^\#$ on~$\tg$ defined~by
\[ f^\#(g) := f(g0) J_g(0)^{-\nu/p}, \qquad g\in\tg,  \label{WC}  \]
satisfies
\[ (\unu_g f)^\# = f^\# \circ g^{-1}  \label{WE}  \]
(i.e.~the map $f\mapsto f^\#$ intertwines the representation \eqref{WB} with the left regular
representation of $\tg$ on~$L^2(\tg)$) and
\[ f^\#(gk) = f^\#(g) J_k^{-\nu/p}, \qquad g\in\tg, \;k\in\tk.  \label{WD}  \]
(Note that $J_k\equiv J_k(0)$ is a constant function, so~we will write just $J_k$ instead
of $J_k(0)$ or~$J_k(z)$.) Furthermore, $f\mapsto f^\#$ is a unitary isomorphism of $L^2(\Omega,d\mu_\nu)$
onto the subspace $L^2(\tg,\nu)$ of all functions in $L^2(\tg/Z(\tg))$ satisfying the transformation rule~\eqref{WD};
the inverse of the map $f\mapsto f^\#$ is given by $F\mapsto F^\flat$, where
\[ F^\flat(g0) := F(g) J_g(0)^{\nu/p}   \label{WG}  \]
(the~right-hand side depends only on~$g0$, thanks to~\eqref{WD}).
See~Proposition~2.1 in \cite{DOZ} for the proof of all these facts.
(Note: there is a misprint in the first formula of Section~2 in~\cite{DOZ},
the $\tau_\nu$ there should be~$\tau_{-\nu}$.)

Using the above identification, one~can view also $\cnmv$ as a subspace of~$L^2(\tg,\nu)$
invariant under the left regular representation~\eqref{WE}. Note~that radial functions on~$\Omega$,
i.e.~those satisfying $f(kz)=f(z)$ for all $z\in\Omega$ and $k\in K$, correspond to functions~$f^\#$
on $\tg$ satisfying
\[ f^\#(k'gk) = J_{k'}^{-\nu/p} f^\#(g) J_k^{-\nu/p}, \qquad g\in\tg, \; k,k'\in\tk. \label{WF}  \]
Such functions on $\tg$ are called \emph{$\nu$-spherical}.

The~representation theory for the space $L^2(\tg,\nu)$ has been developed by Shimeno~\cite{Shm}
(his notation $\tau_{-\ell}(k)$ corresponds to our $J_k^{-\nu/p}$).
(Note that there is a $\tg$-equivariant isomorphism between $L^2(\tg,\nu)$ and $L^2(\tg,-\nu)$,
so~we can always assume that $\nu>0$ as we have started~with.)
Namely, let $\tg=\tk AN$ be the Iwasawa decomposition of~$\tg$ and let $\g$, $\mathfrak k$,
$\mathfrak p$ and $\a$ be the Lie algebras of $G$ (and~$\tg$), $K$ (and~$\tk$), $AN$ and~$A$,
respectively, so~that $\g=\mathfrak k+\mathfrak p$ is the Cartan decomposition of $\g$
and $\a\subset\mathfrak p$ is a maximal Abelian subspace of~$\mathfrak p$.
For $g\in\tg$ let $H(g)$ be the element in the Lie algebra $\a$ of $A$ uniquely determined by $g\in \tk\exp H(g)N$.
Similarly, let $\kappa(g)\in\tk$ be uniquely determined by $g\in\kappa(g)AN$.
For~$\lambda\in\ac$, the complexification of the dual $\a^*$ of~$\a$, one~defines
the \emph{spherical function} $\phin\lambda$ of type~$\nu$~by
\[ \phin\lambda (g) := \int_{\tk/Z(\tg)} e^{-(\lambda+\rho)H(g^{-1}k)} J^{\nu/p}_{k^{-1}\kappa(g^{-1}k)} \,dk, \label{WH} \]
where $Z(\tg)$ denotes the center of~$\tg$, $dk$ is the invariant measure on
the quotient $\tk/Z(\tg)$ with total mass~1, and $\rho\in\a^*$ is the half-sum of positive roots
(see~\eqref{WN} below). Then $\phin\lambda$ is a $\nu$-spherical function on~$\tg$,
and one defines the spherical Fourier transform $\hat f$ of a $\nu$-spherical function $f$ on~$\tg$~by
\[ \hat f (\lambda) := \int_{\tg/Z(\tg)} f(g) \phi_{-\lambda,-\nu}(g) \, dg, \qquad \lambda\in\ac.  \label{WI}  \]
This definition makes sense e.g.~whenever $f$ is compactly supported modulo~$Z(\tg)$.
The~main result of \cite{Shm} then states that there is an inversion formula
$$ f(g) = \int_{\bigcup_{j=0}^r D_{\nu,j}+i\a^*_{\Theta_j}} \hat f(\lambda) \phin\lambda(g) \,d\gamma(\lambda) $$
where $D_{\nu,j}$ and $\a^*_{\Theta_j}$ are certain systems of hyperplanes in~$\a^*$
and $d\gamma$ is a certain measure on~them; and there is also a corresponding Plancherel theorem.
See Theorems~6.7 and~6.8 in \cite{Shm} for the details.
Both $D_{\nu,j}$ and $\a^*_{\Theta_j}$ have codimension $j$ in~$\a^*$;
in~particular, for $j=r$, $\a^*_{\Theta_r}=\{0\}$ and
\[ D_{\nu,r} = \{ \lambda_\m: \m \text{ is a signature with } m_1<\tfrac{\nu-p+1}2 \}  \label{WJ} \]
where
\[ \lambda_\m := \frac12 \sum_{j=1}^r \lambda_j\beta_j, \qquad \lambda_{r+1-j}=p-1-\nu-(j-1)a+2m_j. \label{WT} \]
(Here $\beta_1,\dots,\beta_r\in\a^*$ are the long roots of the root system of~$\Omega$; see~below.)
Thus $D_{\nu,r}$ is a finite discrete set in~$\a^*$. Furthermore, the Plancherel measure $d\gamma$
on $D_{\nu,r}+i\a^*_{\Theta_r}=D_{\nu,r}$ reduces just to a multiple $d_r(\lambda,\nu)\delta_\lambda$
of the Dirac mass at each $\lambda=\lambda_\m$. Here $d_r(\lambda,\nu)$ is given by an explicit
expression involving $\Gamma$-functions; see (4.18), (4.19), (4.21) and (6.18) in~\cite{Shm}.
Altogether, it~thus follows that the space $\amn(\tg)$ spanned in $L^2(\tg)$ by $\tg$-translates
of $\phin{\lambda_\m}$, for~each signature~$\m$, is~an irreducible direct summand of~$L^2(\tg)$,
and as $\m$ varies these summands are mutually orthogonal.
Such summands are called \emph{relative discrete series} representations of~$\tg$.
Note that by~\eqref{WJ}, only signatures $\m$ with $m_1<\frac{\nu-p+1}2$ occur
(in~particular, there are no relative discrete series representations if $\nu\le p-1$).

Let
$$ \amn(\Omega) := \{ F^\flat: F\in\amn(\tg) \}  $$
be~the space of functions on $\Omega$ corresponding to $\amn(\tg)$ via~\eqref{WG}.
It~is then, next, the~central result of \cite{GZkyoto} that $\amn(\Omega)$ is actually
a~space of nearly holomorphic functions of order~$|\m|$, and that, as~$\m$ varies,
these spaces exhaust all nearly holomorphic functions.
Namely, let $N^\m$ be the conical polynomial on $\Omega$ from~\eqref{UN};
one can then form the composition $N^\m(\oz^{\oz})$ with the quasi-inverse~\eqref{VF},
which gives a function on~$\Omega$. By~Theorem~4.7 in~\cite{GZkyoto},
the~space spanned by $\unu_g N^\m(\oz^{\oz})$, $g\in\tg$, coincides with~$\amn(\Omega)$.

Finally, arguing as in the beginning of Section~3 in~\cite{GZstud}, it~follows from Shimeno's
Plancherel formula that the reproducing kernel $A^\m_\nu(z,w)$ of the space $\amn(\Omega)$
is given for $w=0$ (i.e.~at the origin) simply by the appropriate multiple of the spherical function:
\[ A^\m_\nu(z,0) = d_r(\lambda_\m,\nu) \phin{\lambda_\m}^\flat(z).  \label{WK}  \]
Summarizing the discussion so~far, we~have thus arrived at the following result.

\begin{theorem}\label{PH}
The~nearly-holomorphic reproducing kernel $N^m_\nu$ at the origin is given~by
\[ N^\m_\nu(z,0) = \sum_{\substack{|\m|<m\\m_1<\frac{\nu-p+1}2}} d_r(\lambda_\m,\nu) \phin{\lambda_\m}^\flat(z), \label{WO} \]
with $d_r(\lambda,\nu)$, $\phin\lambda$ and $\lambda_\m$ as above.
\end{theorem}

We~pause to note that actually $d_r(\lambda,\nu)=\|\phin\lambda\|^{-2}_{L^2(\tg)}$,
cf.~Remark~6.9 in~\cite{Shm}.

We~conclude by recalling the relation between the spherical functions $\phin\lambda$ and
the hypergeometric functions of Heckman and Opdam \cite[Part~I]{HS}. With our notation $\g$
for the Lie algebra of~$G$, let $\Sigma$ be the restricted root system of the pair~$(\g,\a)$.
Then $\Sigma$ is a root system of type~BC, i.e.~has the form
$$ \Sigma = \{ \pm\tfrac12\beta_j,\pm\beta_j,\pm\tfrac12(\beta_j\pm\beta_k),
 \; 1\le j,k\le r, \; j\neq k \} ,  $$
where $\{\beta_j\}_j$ is a certain basis of~$\a^*$. Here the \emph{short roots} $\pm\frac12\beta_j$
have multiplicity $m_S=2b$, the \emph{long roots} $\pm\beta_j$ have multiplicity~$m_L=1$,
and the \emph{middle roots} $\frac12(\pm\beta_j\pm\beta_k)$ have multiplicity $m_M=a$;
if~$b=0$, then the short roots are actually absent (and~if $r=1$, then the middle roots
are actually absent). The~positive roots are $\frac12\beta_j$, $\beta_j$, $j=1,\dots,r$,
and $\frac12(\beta_j\pm\beta_k)$, $1\le k<j\le r$; the \emph{half-sum of positive roots}
is thus given~by
\[ \rho = \sum_{j=1}^r \frac{b+1+(j-1)a}2 \,\beta_j.   \label{WN}  \]
Let now $F(\lambda,\bk_\nu,\cdot)$ be the Heckman-Opdam hypergeometric function with parameter
$\lambda\in\ac$ corresponding to the root system $2\Sigma$ with multiplicities $\bk_\nu$
given~by
\[ \bk_{\nu,S}=\frac{m_S}2-\nu=b-\nu, \qquad \bk_{\nu,L}=\frac{m_L}2+\nu=\frac12+\nu,
 \qquad \bk_{\nu,M} = \frac{m_M}2=\frac a2 ,  \label{WS} \]
for the ``doubles'' of the short, long and middle roots of~$\Sigma$, respectively.
Then for any $g\in A\subset\tg$ we~have
\[ \phin\lambda(g) = h(g0,g0)^{-\nu/2} F(\lambda,\bk_\nu,g).  \label{WL}  \]
See~\cite[Theorem~5.2.2]{HS} (cf.~also Remark~3.8 in~\cite{Shm}).
Note further that by~\cite[(4.4.10)]{HS} (cf.~also Remark~5.12 in~\cite{Shm}),
if~$\lambda$ is a dominant weight (i.e.~$\frac{\spr{\lambda,\alpha}}{\spr{\alpha,\alpha}}\in\NN$
for all positive roots $\alpha$ of~$\Sigma$), then
\[ F(\lambda,\rho_\nu,\bk_\nu,\cdot) = c(\lambda+\rho_\nu,\bk_\nu) P(\lambda,\bk_\nu,\cdot), \label{WM} \]
where $\rho_\nu$ is given by \eqref{WN} with $b$ replaced by $b+\nu$, and $P(\lambda,\bk_\nu,\cdot)$
are the multivariable Jacobi polynomials (cf.~Debiard \cite{DebSh3}, \cite{DebSh4}).
Here $c(\lambda,\bk)$ is the generalized $c$-function of Harish-Chandra with line bundle parameter~$\nu$
(cf.~(3.15) in~\cite{Shm} or (3.4.3) in~\cite{HS}).

Combining \eqref{WO}, \eqref{WL} and \eqref{WN} with Theorem~\ref{PF}, we~can express also
the invariantly-polyanalytic reproducing kernels $P^m_\nu$ and the reproducing kernels $S^m_\nu$
of the spaces $\csmv$ of symmetric polynomials on~$\rpr$ in terms of spherical functions,
or Heckman-Opdam hypergeometric functions, or multivariable Jacobi polynomials.
We~omit the details.

\begin{example}\label{PI}
For $\Omega=\BB^d$, the unit ball of~$\CC^d$, we~have $r=1$, $b=d-1$, $a$~is not defined, $p=d+1$
and $h(z,w)=1-\spr{z,w}$. The~elements of the group $G=SU(1,n)$ can be identified with
$(n+1)\times(n+1)$ complex matrices $\left(\begin{smallmatrix}A&B\\C^t&D\end{smallmatrix}\right)$,
with $A\in\CC$, $B,C\in\CC^{1\times n}$ and $D\in\CC^{n\times n}$, acting by
$z\mapsto(Az+B)(Cz+D)^{-1}$, with $z\in\BB^d$ written as row vector.
The~Lie algebra $\a$ equals $\RR H$, where
$$ H = \begin{pmatrix} 1 & 0^{1\times(n-1)}&0\\0^{(n-1)\times 1}&0^{(n-1)\times(n-1)}&0^{(n-1)\times 1}\\0&0^{1\times(n-1)}&1\end{pmatrix}, $$
and defining $\beta\in\a^*$ by $\beta(H)=2$ the root system is given by $\Sigma=\{\pm\frac12\beta,\pm\beta\}$.
We~have
$$ \exp(tH)= \begin{pmatrix} \cosh t & 0^{1\times(n-1)}& \sinh t \\0^{(n-1)\times 1}&I^{(n-1)\times(n-1)}&0^{(n-1)\times 1}\\ \sinh t &0^{1\times(n-1)}& \cosh t \end{pmatrix}, $$
hence $|\exp(tH)0|=|\tanh t|$ and
\[ \cosh t = h(\exp(tH)0,\exp(tH)0)^{-1/2}.  \label{WQ}  \]
The~spherical functions $\phin\lambda$ are given~by (\cite[(8.2) and (8.3)]{Shm})
\begin{align}
\phin\lambda (\exp tH) &= (\cosh t)^{-\nu} \F21{\frac{d-\nu+\lambda}2,\frac{d-\nu-\lambda}2}d{-\sinh^2 t} \nonumber \\
&= (\cosh t)^\nu \F21{\frac{d+\nu+\lambda}2,\frac{d+\nu-\lambda}2}d{-\sinh^2 t} , \label{WP}
\end{align}
that~is, by~\eqref{WG} and~\eqref{WQ}, and since $J_{\exp tH}(0)=\cosh^{-p}t$,
\begin{align*}
\phin\lambda^\flat(z) &= (1-|z|^2)^\nu \F21{\frac{d-\nu+\lambda}2,\frac{d-\nu-\lambda}2}d{\frac{|z|^2}{|z|^2-1}} \\
&= \F21{\frac{d+\nu+\lambda}2,\frac{d+\nu-\lambda}2}d{\frac{|z|^2}{|z|^2-1}} ,
\end{align*}
where ${}_2\!F_1$ is the Gauss hypergeometric function and on the right-hand sides, we~write just
$\lambda$ for $\lambda(H)$. Using the standard transformation formula for~${}_2\!F_1$ \cite[\S2.1~(22)]{BE}
\[ \F21{a,b}cz = (1-z)^{-a}\F21{a,c-b}c{\frac z{z-1}} = (1-z)^{-b} \F21{c-a,b}c{\frac z{z-1}}, \label{WR}  \]
this can also be written~as
\begin{align*}
\phin\lambda^\flat(z) &= (1-|z|^2)^{\frac{d+\lambda+\nu}2} \F21{\frac{d+\lambda+\nu}2,\frac{d+\lambda-\nu}2}d{|z|^2} \\
&= (1-|z|^2)^{\frac{d-\lambda+\nu}2} \F21{\frac{d-\lambda+\nu}2,\frac{d-\lambda-\nu}2}d{|z|^2} .
\end{align*}
The~elements $\lambda_\m\in\a^*$, $\m=(m_1)$, are given~by $\lambda_\m=\lambda_1\frac\beta2$
with $\lambda_1=2m_1+d-\nu$, $m_1\in\ZZ$, $0\le m_1<\frac{\nu-d}2$.
The~corresponding space $\amn(\BB^d)$ is spanned by $G$-translates of the function
$\oz_1^{m_1}(1-|z|^2)^{-m_1}$ under the action~\eqref{TL} \cite[Section~5]{GZkyoto}
and coincides with the orthogonal complement $\cN^{m_1+1}_\nu\ominus\cN^{m_1}_\nu$ \cite[pp.~116-117]{GZstud}.
The~reproducing kernel of $\amn(\BB^d)$ at the origin is given by~\eqref{WO} with
$$ d_1(\lambda_\m,\nu) = \frac{(\nu-d-2m_1)\pi^{-d}\Gamma(d+m_1)\Gamma(\nu-m_1)}{m_1!\Gamma(d)\Gamma(\nu-d+1-m_1)} ,  $$
in~complete agreement with \cite[Section~3, bottom of p.~116]{GZstud}.
(Note that there is a misprint in the formula (1.5) in~\cite{GZstud}: the $\Gamma(\alpha+1+d)$
in the numerator should be $\Gamma(\alpha+1+d-l)$. Also the labeling of spherical functions
is different there: our~$\phin\lambda$ corresponds to $\phin{i\lambda}$ in~\cite{GZstud}.)

By~\cite[p.~90]{Opd}, the Heckman-Opdam hypergeometric function for root system BC$_1$ is given~by
$$ F(\lambda,\bk,\exp tH) = \F21{\tfrac{d+\nu+\lambda}2,\tfrac{d+\nu-\lambda}2}d{-\sinh^2 t}, $$
in~full accordance with \eqref{WL} and~\eqref{WP} in~view of~\eqref{WQ}.

Finally, for rank one \eqref{WM} reduces just to~\eqref{TN}, and thus for $m-1<\frac{\nu-p+1}2$
\begin{align*}
N^m_\nu(z,0) &= \sum_{m_1=0}^{m-1} d_1(\lambda_\m,\nu) \F21{-m_1,d-\nu+m_1}d{\frac{|z|^2}{|z|^2-1}} \\
&= \sum_{m_1=0}^{m-1} d_1(\lambda_\m,\nu) (1-|z|^2)^{-m_1} \F21{-m_1,\nu-m_1}d{|z|^2} \qquad\text{by \eqref{WR}} \\
&= \sum_{m_1=0}^{m-1} d_1(\lambda_\m,\nu) \frac{(1-|z|^2)^{-m_1}}{\binom{m_1+d-1}{d-1}} P^{(d-1,\nu-d-2m_1)}_{m_1}(1-2|z|^2),
\end{align*}
recovering~\eqref{TO} and its reformulation in terms of Jacobi polynomials.

From Theorem~\ref{PF} we see that the reproducing kernel at the origin for the subspace
of all polynomials of degree $<m$ in $L^2(\RR_+,c_{\BB^d}t^{d-1}(1+t)^{-\nu}\,dt)$ equals
\[ S^m_\nu(x,0) = \sum_{0\le m_1<\min(m,\frac{\nu-d}2)} d_1(\lambda_{(m_1)},\nu) \; \F21{-m_1,d-\nu+m_1}d{-x}, \label{WU} \]
the summands being actually mutually orthogonal.  \qed
\end{example}

\section{Faraut-Koranyi hypergeometric functions}
With Theorem~\ref{PF} in~mind, let~us return to the reproducing kernels $S^m_\nu(x,y)$
of the subspaces $\csmv$ of symmetric polynomials of degree $<m$ in $L^2(\rpr,d\rho_{b,\nu,a})$.
By~Propositions~\ref{PD} and~\ref{PE}, the~functions
\[ K_\m(xe,e), \qquad |\m|<m, \; m_1<\tfrac{\nu-p+1}2,  \label{XA}  \]
span~$\csmv$. The~following easy fact --- capturing, in~effect, the standard Gram-Schmidt
orthogonalization process --- describes how to extract the reproducing kernel from an arbitrary basis.

\begin{proposition}\label{PJ}
Let $\HH$ be a finite-dimensional Hilbert space of functions with (not~necesarily orthogonal) basis~$\{f_j\}$.
Denote by $\UU(x)$ the column vector $(f_j(x))_j$. Then the reproducing kernel of~$\HH$ is given~by
\[ K_\HH (x,y) = \UU(y)^* \GG^{-1} \UU(x),  \label{XC} \]
where $\GG=(\spr{f_j,f_k}_\HH)_{j,k=1}^{\dim\HH}$ is the Grammian matrix of the basis~$\{f_j\}_j$.
\end{proposition}

\begin{proof} Let $\{e_l\}_l$ be an orthonormal basis of $\HH$ and let $e_l=\sum_j c_{lj}f_j$
be the expressions of $e_l$ as linear combinations of the~$f_j$. Let $C$ denote the matrix~$(c_{lj})_{l,j=1}^{\dim\HH}$.
From
$$ \delta_{lm} = \spr{e_l,e_m}_\HH = \sum_{j,k} c_{lj}\overline{c_{mk}} \spr{f_j,f_k}_\HH = (C\GG C^*)_{lm}  $$
we~see that $C\GG C^*$ is the identity matrix; that~is, $\GG=(C^*C)^{-1}$. Hence by~\eqref{VA}
\begin{align*}
K_\HH(x,y) &= \sum_l e_l(x)\overline{e_l(y)} = \sum_{l,j,k} c_{lj} f_j(x) \overline{c_{lk} f_k(y)} \\
&= \UU(y) C^*C \UU(x) = \UU(y) \GG^{-1} \UU(x),
\end{align*}
proving the claim.
\end{proof}

For~the basis~\eqref{XA}, the Grammian matrix $\GG$ is in principle easy to compute explicitly for
low values of $m$ and~$r$. For~instance, for $r=2$, starting from
$$ (1-t)^n = h(te,e)^n = \sum_{\m: m_1\le n} (-n)_\m K_\m(te,e), \qquad n=0,1,2,\dots,  $$
one recursively reads off $(-n)_{(n,m_2)}K_{(n,m_2)}(te,e)$ as the homogeneous component
of degree $n+m_2$ in~$(1-t)^n$:
\begin{gather*}
K_{(0)}=\jedna, \qquad K_{(1)}=t_1+t_2, \qquad K_{(1,1)}=\frac{2t_1t_2}{a+2}, \\
K_{(2,0)}=\frac{t_1^2+t_2^2}2+\frac{at_1t_2}{a+2}, \qquad K_{(2,1)}=\frac{2t_1t_2(t_1+t_2)}{a+4}, \\
K_{(2,2)}=\frac{2t_1^2t_2^2}{(a+2)(a+4)} , \qquad \dots\;.
\end{gather*}
(For~brevity, we~have omitted the arguments $(te,e)$.) This reduces the computation of $\GG$
to evaluation of the integrals
\[ \int_0^\infty \int_0^\infty x_1^{q_1} x_2^{q_2} (1+x_1)^{-\nu}(1+x_2)^{-\nu} |x_1-x_2|^a \,dx_1\,dx_2. \label{XB} \]
For $a$ an even nonnegative integer, the last integral can be evaluated by expanding $(x_1-x_2)^a$
via the binomial theorem and integrating term by term using the standard formula
$$ \int_0^\infty \frac{x^q}{(1+x)^\nu}\,dx = \frac{\Gamma(q+1)\Gamma(\nu-q-1)}{\Gamma(\nu)} \equiv B(q+1,\nu-q-1),
 \qquad -1<q<\nu-1,  $$
for the Beta integral. The~outcome is that \eqref{XB} equals
$$ \sum_{j=0}^a \frac{(-a)_j}{j!} B(q_1+j+1,\nu-q_1-a-1) B(q_2+a-j+1,\nu-q_2-a-1), \qquad a\in2\NN.  $$
Taking for $\UU(x)$ the column vector $(K_\m(xe,e))_{|\m|<m,m_1<(\nu-p+1)/2}$, one~can then use \eqref{XC}
to obtain a formula for $S^m_\nu(x,0)$.

For $a\notin2\NN$, a~possible way of evaluating \eqref{XB} is first making the change of variable $x=\frac t{1-t}$,
which transforms \eqref{XB} into
\[ \begin{aligned}
& \int_0^1 \int_0^1 t_1^{q_1} t_2^{q_2} (1-t_1)^{\nu-a-2-q_1} (1-t_2)^{\nu-a-2-q_2} |t_1-t_2|^a \,dt_1\,dt_2 \\
&\hskip4em= \int_0^1 \int_0^1 (1-t_1)^{q_1} (1-t_2)^{q_2} t_1^{\nu-a-2-q_1} t_2^{\nu-a-2-q_2} |t_1-t_2|^a \,dt_1\,dt_2.
\end{aligned}  \label{XH}  \]
Introducing temporarily the notation
$$ I(\alpha,\beta,\gamma,\delta):=\int_0^1 \int_0^{t_1} (1-t_1)^\alpha (1-t_2)^\beta t_1^\gamma t_2^\delta |t_1-t_2|^a \,dt_2\,dt_1,  $$
\eqref{XH} thus equals
$$ I(q_1,q_2,\nu-a-2-q_1,\nu-a-2-q_2)+I(q_2,q_1,\nu-a-2-q_2,\nu-a-2-q_1). $$
Now~making the change of variable $t_2=yt_1$ yields
\begin{align}
I(\alpha,\beta,\gamma,\delta) &= \int_0^1 \int_0^1 (1-t_1)^\alpha (1-yt_1)^\beta t_1^\gamma (y t_1)^\delta t_1^a(1-y)^a \,t_1\,dy \,dt_1  \nonumber \\
&= \sum_{j=0}^\infty \frac{(-\beta)_j}{j!} \int_0^1 \int_0^1 (1-t_1)^\alpha (yt_1)^j t_1^\gamma (y t_1)^\delta t_1^a(1-y)^a \,t_1\,dy \,dt_1  \nonumber \\
&= \sum_{j=0}^\infty \frac{(-\beta)_j}{j!} B(\alpha+1,j+\gamma+\delta+a+2) B(a+1,j+\delta+1) . \label{XI}
\end{align}
For~$\beta\in\NN$ --- which is our case in~\eqref{XH} --- the~series terminates, and one thus
has an expression for~\eqref{XB}, albeit the formula is a bit more unwieldy than the one from the previous paragraph.

Similarly, for rank~3, recall that
$$ K_\m(te,e) = \frac{\pi_\m}{(\qo)_\m} \phi_\m(te),  $$
with $\pi_\m$ and $\qo$ given by \eqref{tag:XP} and~\eqref{tag:QO}, respectively,
and $\phi_\m$ the spherical polynomial corresponding to the signature~$\m$.
Using the formula~\eqref{XK}, one~can again successively read off $K_\m(te,e)$
as the term of the appropriate homogeneity degree~in
$$ (1-t)^n = h(te,e)^n = \sum_{\m: m_1\le n} (-n)_\m \frac{\pi_\m}{(\qo)_\m} \phi_\m(te), \qquad n=0,1,2,\dots.  $$
This yields (omitting again the argument~$te$)
\begin{gather*}
\phi_{(0)}=\jedna, \qquad \phi_{(1)}=\frac{t_1+t_2+t_3}3, \qquad \phi_{(1,1)}=\frac{t_1t_2+t_1t_3+t_2t_3}3, \\
\phi_{(1,1,1)}=t_1t_2t_3, \qquad \phi_{(2)}=\frac{(a+2)(t_1^2+t_2^2+t_3^2)}{3(3a+2)} + \frac{2a(t_1t_2+t_1t_3+t_2t_3)}{3(3a+2)}, \\
\phi_{(2,1)}=\frac{(a+1)(t_1^2t_2+t_1^2t_3+t_2^2t_1+t_2^2t_3+t_3^2t_1+t_3^2t_2)}{3(3a+2)} + \frac{3at_1t_2t_3}{3(3a+2)}, \\
\phi_{(2,1,1)}=\frac{(t_1+t_2+t_3)t_1t_2t_3}3, \\
\phi_{(2,2)}=\frac{(a+2)(t_1^2t_2^2+t_1^2t_3^2+t_2^2t_3^2)}{3(3a+2)} + \frac{2a(t_1+t_2+t_3)t_1t_2t_3}{3(3a+2)}, \\
\phi_{(2,2,1)}=\frac{(t_1t_2+t_1t_3+t_2t_3)t_1t_2t_3}3, \qquad \phi_{(2,2,2)}=t_1^2 t_2^2 t_3^2, \qquad \dots\; .
\end{gather*}
This once more reduces the computation of $\GG$ to the evaluation of the three-variable analogue of~\eqref{XB},
which for $a\in2\NN$ is again by the same ``bare hands'' method seen to be equal~to
\begin{multline*}
 \sum_{j,k,l=0}^a \frac{(-a)_j(-a)_k(-a)_l}{j!k!l!} B(q_1+1+j+k,\nu-2a-1-q_1) \\
 \quad \times B(q_2+1+a-j+l,\nu-2a-1-q_2) \\
 \quad \times B(q_3+1+2a-k-l,\nu-2a-1-q_3), \qquad a\in2\NN.
\end{multline*}
For $a\notin2\NN$, one~can again proceed as~for~\eqref{XI}, but the outcome is quite cumbersome.

Carrying out all these calculations leads to the following conjecture.

Recall that for $\alpha,\beta,\gamma\in\CC$, the \emph{Faraut-Koranyi hypergeometric function}
on $\Omega$ with parameters $\alpha,\beta,\gamma$ is defined~by \cite{FK88}
\[ \FKo21{\alpha,\beta}\gamma z := \sum_\m \frac{(\alpha)_\m(\beta)_\m}{(\gamma)_\m} K_\m(z,z).  \label{XOM}  \]
Here $\gamma$ is assumed to be such that $(\gamma)_\m\neq0$ $\forall\m$.
Alternatively, one~sometimes views these just as symmetric functions~on~$\rpr$ \cite{Yan}:
\[ \FK21{\alpha,\beta}\gamma t := \sum_\m \frac{(\alpha)_\m(\beta)_\m}{(\gamma)_\m} K_\m(te,e),  \label{XN}  \]
the two variants being related simply by ${}_2\!{\mathcal F}^\Omega_1(z)={}_2\!{\mathcal F}_1(t)$ for $z=k\sqrt te$.

\begin{conjecture}\label{PK} Assume that $m\ge rq+1$ where $q$ is the nonnegative integer such that
$q<\frac{\nu-p+1}2\le q+1$. Then the reproducing kernel $S^m_\nu$ on $\rpr$ at the origin is given~by
\[ S^m_\nu(x,0) = c^q_\nu \FK21{-q,-b-\nu+2p+q-2r}p{-x}, \label{XD}  \]
where
\[ c^q_\nu = \frac{\gom(\nu-p-q+2r+b) \gom(\qo) \gom(p+q)} {\pi^d \gom(\nu-p-q+2r-\qo) \gom(\qo+q) \gom(p)} . \label{XOC} \]
\end{conjecture}

\bigskip
The~last conjecture holds for $r=1$, by~\eqref{TG}, \eqref{TN}, Theorem~\ref{PF} and~\eqref{WR}.
It~has also been verified by computer for
\begin{align*}
& r=2, \; q\in\{0,1,2\}, \; a\in\{1,2,3,4,5,6,7,8\}, \; b\in\{0,1,2,3\}, \text{ $\nu$ arbitrary,}  \\
& r=2, \; q=3, \; a\in\{1,2,3,4\}, \; b\in\{0,1,2,3\}, \text{ $\nu$ arbitrary,}  \\
& r=3, \; q\in\{0,1,2\}, \; a\in\{2,4\}, \; b\in\{0,1,2,3\}, \text{ $\nu$ arbitrary,} \\
& r=3, \; q\in\{0,1,2\}, \; a=8, \; b=0, \text{ $\nu$ arbitrary,}
\end{align*}
and a couple more values of $a$, $b$ and $\nu$ for $r\in\{2,3\}$ and $q\in\{0,1,2\}$.
(Note that the above values of $r,a,b$ include, in~particular, both exceptional bounded symmetric domains of dimensions~16 and~27.)

Note that the hypothesis of the conjecture, that~is,
\[ q-1 < \frac{\nu-p-1}2 \le q \le \frac{m-1}r, \label{HYPO}  \]
corresponds precisely to the case \eqref{CASE2} of the ``stabilized'' kernels from Corollary~\ref{COCO}.
Without this hypothesis, the conjecture fails, as~the following example shows.

\begin{example} Let $r=2$, $m=2$ and $\nu-p>1$
(note that this corresponds to the case \eqref{CASE1} in Corollary~\ref{COCO}). The~space $\csmv$ is thus spanned by $K_{(0)}(xe,e)=\jedna$
and $K_{(1)}(xe,e)=x_1+x_2$. Performing the calculations outlined above yields
$$ \tfrac1C S^m_\nu (x,0) = K_{(0)} + \frac{(b-\nu+a+3)(2b-2\nu+a+4)}
 {a^2 + (7+4b-2\nu)a + (4b^2-4 b\nu+16b-6\nu+14)} K_{(1)}  $$
with some constant~$C$. (We~have omitted the arguments $(xe,e)$ at $K_{(0)}$ and~$K_{(1)}$.)
Plainly, the right-hand side is not of the form~${}_2{\mathcal F}_1$.  \qed
\end{example}

In~the remaining case from Corollary~\ref{COCO} (i.e.~$q+1<m<rq+1$), the kernels can be expected to
be even more ``ugly'' than in the last example.

By~the results of the preceding sections, the~validity of the conjecture would have the following consequences.

\begin{corollary}\label{PO} {\rm (Subject to Conjecture~\ref{PK})}
Assume that $m\ge rq+1$ where $q$ is the nonnegative integer such that $q<\frac{\nu-p+1}2\le q+1$.
Then the nearly-holomorphic reproducing kernel $N^m_\nu$ at the origin is given~by
$$ N^m_\nu(k\sqrt te,0) = c^q_\nu \FK21{-q,-b-\nu+2p+q-2r}p{\frac t{t-1}} , $$
or
$$ N^m_\nu(z,0) =  c^q_\nu h(z,z)^{-q} \FKo21{-q,b+\nu-p-q+2r}pz . $$
\end{corollary}

\begin{proof} By~Theorem~\ref{PF},
$$ N^m_\nu(k\sqrt te,0) = S^m_\nu(\tfrac t{1-t},0),  $$
and~\eqref{XD} gives the first formula. The~second formula then follows from the Kummer relation
(a~counterpart of~\eqref{WR} for the ordinary~${}_2\!F_1$)
\[ \FK21{\alpha,\beta}\gamma t = (1-t)^{-\alpha} \FK21{\alpha,\gamma-\beta}\gamma{\frac t{t-1}} , \label{XL} \]
see~\cite[formula (35)]{Yan}.
\end{proof}

\begin{theorem} \label{PL}  {\rm (Subject to Conjecture~\ref{PK})}
Assume that $m\ge rq+1$ where $q$ is the nonnegative integer such that $q<\frac{\nu-p+1}2\le q+1$.
Then with the notation \eqref{WS}, \eqref{WT} and~\eqref{XOM},
\[ \begin{aligned}
& \sum_{|\m|:\;m_1\le q} d_r(\lambda_\m,\nu) F(\lambda_\m,\bk_\nu,g) = \\
& \hskip4em h(z,z)^{-q} c^q_\nu \FKo21{-q,b+\nu-p-q+2r}pz \end{aligned} \label{XE}  \]
for $z=g0$ with $g\in A$.
\end{theorem}

\begin{proof} Since $J_g(0)=g(g0,g0)^{p/2}$ for $g\in A$, \eqref{WG}~and~\eqref{WL} yield
$$ \phin\lambda^\flat(g0) = F(\lambda,\bk_\nu,g0).  $$
Thus by~\eqref{WO}, the left-hand side of \eqref{XE} equals $N^m_\nu(g0,0)$.
By~Corollary~\ref{PO}, the~latter is precisely the right-hand side of~\eqref{XE}.
\end{proof}

Note that for $r=1$, \eqref{XE}~recovers the formula \eqref{TM} from the Introduction.

\begin{remark}
By~Theorem~4.2 of Beerends and Opdam~\cite{BO}, $F(\lambda,\bk_\nu,\cdot)$ for the special value
$$ \lambda=-\alpha\sum_j\beta_j+\rho_\nu, \qquad \alpha\in\CC, $$
can~be expressed in terms of
$$ \FK21{\alpha,d/r+\nu-\alpha}{d/r}{\cdot}; $$
however the ${}_2\!{\mathcal F}_1$ in~\eqref{XE} does not seem reducible to this form.  \qed
\end{remark}

Using again Theorem~\ref{PF}, the conjecture also implies a formula for the invariantly-polyanalytic kernel~$S^m_\nu(x,0)$.

\begin{corollary} \label{PM} {\rm (Subject to Conjecture~\ref{PK})}
Assume that $m\ge rq+1$ where $q$ is the nonnegative integer such that $q<\frac{\nu-p-1}2+m\le q+1$.
Then the reproducing kernel $P^m_\nu$ at the origin is given~by
$$ P^m_\nu(z,0) = c^m_{\nu+2m-2} h(z,z)^{m-1-q} \FK21{-q,b+\nu+2m-2-p-q+2r}pz . $$
\end{corollary}

\begin{proof} By~Theorem~\ref{PF}, $P^m_\nu(z,0)=h(z,z)^{m-1}N^m_{\nu+2m-2}(z,0)$,
and the claim follows by Corollary~\ref{PO}.
\end{proof}

\begin{example} \label{PN}
Continuing our example of $\Omega=\BB^d$ from the previous section, for rank~1 the Faraut-Koranyi
hypergeometric function coincides with the ordinary Gauss hypergeometric function
$$ {}_2\!{\mathcal F}^{\BB^d}_1\Big(\begin{matrix}\alpha,\beta\\\gamma\end{matrix}\Big|z\Big) = \F21{\alpha,\beta}\gamma{|z|^2} . $$
By~\eqref{WU} and~\eqref{XD} we therefore~get, for $0\le q<\frac{\nu-d}2\le q+1\le m$,
$$ \sum_{j=0}^q d_1(\lambda_{(j)},\nu) \F21{-j,d-\nu+j}d{-x}
 = c^q_\nu \F21{-q,d+q+1-\nu}{d+1}{-x}. $$
This is, of~course, just \eqref{TM} in disguise.  \qed
\end{example}

\begin{remark} \label{HELGASON}
The~formula \eqref{WU} actually shows that $d_1(\lambda_{(j)},\nu) \F21{-j,d-\nu+j}d{-x}$
is the reproducing kernel of the orthogonal complement $\cS^j_\nu\ominus\cS^{j-1}_\nu$,
$0\le j<\frac{\nu-d}2$ (with $\cS^{-1}_\nu:=\{0\}$).
Theorem~V.4.5 in Helgason~\cite{He} identifies the last ${}_2\!F_1$ as the spherical function
for the compact dual $SU(d+1)/SU(d)=\CC P^d$ of~$\BB^d$.   \qed
\end{remark}

By~the reproducing property, Conjecture~\ref{PK} is equivalent~to
\[ \begin{aligned}
 & \int_{\rpr} K_\m(xe,e) \; \FK21{-q,-b-\nu+2p+q-2r}p{-x} \, d\rho_{b,\nu,a}(x) \\
 & \hskip8em = \frac1{c^q_\nu} \delta_{\m,(0)}, \qquad \forall\m\text{ with }m_1\le q,
\end{aligned} \label{XM} \]
where $q$ is the nonnegative integer such that $q<\frac{\nu-p+1}2\le q+1$ and $m\ge rq+1$.
Taking in particular $m=(0)$ yields
$$ \frac1{c^q_\nu} = \int_{\rpr} \FK21{-q,-b-\nu+2p+q-2r}p{-x} \, d\rho_{b,\nu,a}(x)  $$
(subject to the validity of Conjecture~\ref{PK}). The~last integral can be evaluated explicitly.

\begin{proposition}\label{PP}
\begin{multline*}
 \int_{\rpr} \FK21{-q,-b-\nu+2p+q-2r}p{-x} \, d\rho_{b,\nu,a}(x) \\
 = \frac1{c_{\nu-q}} \; \FKo32{-q,b+\nu-p-q+2r,d/r}{p,\nu-q}e,
\end{multline*}
where the Faraut-Koranyi function ${}_3{\mathcal F}^\Omega_2$ is defined analogously as~in~\eqref{XOM}.
\end{proposition}

\begin{proof} Making again the change of variable $x=\frac t{1-t}$, we~get from \eqref{VM} and~\eqref{XN}
\begin{align*}
& \int_{\rpr} \FK21{-q,-b-\nu+2p+q-2r}p{-x} \, d\rho_{b,\nu,a}(x) \\
& \hskip2em = c_\Omega \int_{[0,1]^r} \FK21{-q,-b-\nu+2p+q-2r}p{\frac t{t-1}} \,d\mu_{b,\nu,a}(t) \\
& \hskip2em = c_\Omega \int_{[0,1]^r} (1-t)^{-q} \FK21{-q,b+\nu-p-q+2r}pt \,d\mu_{b,\nu,a}(t) \\
& \hskip2em = c_\Omega \int_{[0,1]^r} \FK21{-q,b+\nu-p-q+2r}pt \,d\mu_{b,\nu-q,a}(t) \\
& \hskip2em = c_\Omega \sum_{|\m|<m} \frac{(-q)_\m(b+\nu-p-q+2r)_\m}{(p)_\m} \int_{[0,1]^r} K_\m(te,e) \,d\mu_{b,\nu-q,a}(t).
\end{align*}
If~$\{\psi_j\}_{j=1}^{d_\m}$ is an orthonormal basis of~$\PP_\m$ with respect to the Fock norm,
the~last integral equals, by~\eqref{VA},
\begin{align*}
\int_{[0,1]^r} K_\m(\sqrt te,\sqrt te) \,d\mu_{b,\nu-q,a}(t)
&= \int_K \int_{[0,1]^r} K_\m(k\sqrt te,k\sqrt te) \,d\mu_{b,\nu-q,a}(t) \,dk \\
&= \frac1{c_\Omega} \int_\Omega K_\m(z,z) \,d\mu_{\nu-q}(z) \qquad\text{by \eqref{tME}} \\
&= \frac1{c_\Omega} \int_\Omega \sum_j |\psi_j(z)|^2 \,d\mu_{\nu-q}(z) \\
&= \frac1{c_\Omega} \sum_j \|\psi_j\|^2_{\nu-q} \\
&= \frac1{c_\Omega} \sum_j \frac{\|\psi_j\|^2_F}{(\nu-q)_\m c_{\nu-q}} \qquad\text{by \eqref{tAU}} \\
&= \frac{d_\m} {c_\Omega (\nu-q)_\m c_{\nu-q}} .
\end{align*}
Consequently,
\begin{align*}
& \int_{\rpr} \FK21{-q,-b-\nu+2p+q-2r}p{-x} \, d\rho_{b,\nu,a}(x) \\
& \hskip1em = \sum_{|\m|<m} \frac{(-q)_\m(b+\nu-p-q+2r)_\m}{(p)_\m} \frac{d_\m} {(\nu-q)_\m c_{\nu-q}}  \\
& \hskip1em = \sum_{|\m|<m} \frac{(-q)_\m(b+\nu-p-q+2r)_\m}{(p)_\m} \frac{(d/r)_\m} {(\nu-q)_\m c_{\nu-q}} K_\m(e,e) \\
& \hskip1em = \frac1{c_{\nu-q}} \; \FKo32{-q,b+\nu-p-q+2r,d/r}{p,\nu-q}e ,
\end{align*}
as~claimed. Here the second equality is due to~\eqref{tag:XC}.
\end{proof}

The~formula \eqref{XOC} thus gives a conjectured value for this ${}_3{\mathcal F}_2$ function.

For rank~1, we~have $b+\nu-p-q+2r=\nu-q$, so the ${}_3\!F_2$ becomes ${}_2\!F_1$
and~\eqref{XOC} follows by the standard formula for~$\F21{a,b}c1$.

\section{Compact Hermitian symmetric spaces}
We~now consider also the compact duals of Hermitian symmetric spaces~$\hom$, the simplest examples of
these being the complex projective space $\CC P^d$ as the compact dual of the unit ball $\BB^d$
(including, in~particular, the Riemann sphere $\CC P^1$ as the compact dual of the unit disc).
Most results are obtained by formally replacing $\nu$ by~$-\nu$, $h(z,z)$ by $h(z,-z)$, and
$\Omega\subset\CC^d$ by the open chart $\CC^d\subset\hom$. We~shall be rather brief.

The symmetric space $\Omega=G/K$ has its compact dual $\hom=\hg/K$ where $\hg$ is a simply connected
compact Lie group with Lie algebra $\hat{\g}=\mathfrak k+i\mathfrak p$. There is a dense
open subset of $\hom$ that is biholomorphic to~$\CC^d$, and we shall simply identify this
local chart with $\CC^d$ throughout. The~stabilizer subgroup $K$ of the origin in $\hg$
is the same as in the bounded case. For~$x\in\hom$, there is again a unique geodesic symmetry
$\hphi_x\in\hg$ which interchanges $x$ and the origin, i.e.~$\hphi_x\circ\hphi_x=\text{id}$,
$\hphi_x(0)=x$, $\hphi_x(x)=0$, and $\hphi_x$ has only isolated fixed points. Any $g\in\hg$ can
be uniquely written in the form $g=\hphi_xk$ with $k\in K$ and $x=g0\in\hom$. The~measure
$$ d\hmu_\nu(z) := h(z,-z)^{-\nu-p} \, dz  $$
on~$\CC^d\subset\hom$ is finite if and only of $\nu>-1$, and one can again consider the spaces
$$ \hat A_\nu:= L^2(\hom,d\hmu_\nu) \cap \cO(\CC^d) .  $$
The~elements of $\hat A_\nu$ extend to holomorphic sections on all of $\hom$ if and only if
$\nu$ is an \emph{integer}, which we will assume from now on throughout the rest of this section.
In~that case,
$$ \hat A_\nu = \bigoplus_{\m: m_1\le\nu} \PP_\m ,  $$
and $\hat A_\nu$ possesses a reproducing kernel, given~by
$$ \hat K_\nu(z,w) = \hat c_\nu \; h(z,-w)^\nu, \qquad z,w\in\CC^d\subset\hom, \;\nu\in\NN, $$
where
$$ \hat c_\nu = \frac{\gom(\nu+p)}{\pi^d \gom(\nu+p-\frac dr)} .  $$
(Here, as~before, $p$, $r$, $a$ and $b$ denote the genus, the rank, and the characteristic
multiplicities of~$\hom$, which are all the same as for~$\Omega$.)
From the transformation rule
$$ h(\hphi z,-\hphi w) = \frac{h(a,-a)h(z,-w)}{h(z,-a)h(a,-w)}, \qquad a=\hphi^{-1}0, \qquad z,w\in\CC^d, \; \hphi\in\hg,  $$
it~again follows that the measure $d\hmu_0$ is $\hg$-invariant and that $\hat\Psi(z):=-\log h(z,-z)$
is the K\"ahler potential for a $\hg$-invariant Riemannian metric on $\hom$.
We~thus again have the associated Cauchy-Riemann operator~$\oD$, and the corresponding spaces
$\hat{\cN}^m:=\operatorname{Ker}\oD{}^m$ of nearly holomorphic functions on~$\hom$ of order~$m$,
as~well as their Bergman-type subspaces
$$ \hnmv := L^2(\hom,d\hmu_\nu) \cap \operatorname{Ker}\oD{}^m .  $$
One~can also proceed to define the invariantly polyanalytic functions $\hat{\cP}^m$ and their
Bergman-type subspaces $\hat{\cP}^m_\nu$ as in the bounded case.

In~the polar coordinates~(\ref{tQA}), the measures $d\hmu_\nu$ assume the form
\[ \int_{\hom} f(z) \, d\hmu_\nu(z) =
 \quad c_\Omega \int_{\rpr} \int_K f(k\sqrt te) \, dk \,d\hmu_{b,\nu,a}(t),  \]
where
\[ d\hmu_{b,\nu,a}(t) := t^b (1+t)^{-\nu-p} \prod_{1\le i<j\le r} |t_i-t_j|^a \, dt  \]
and $c_\Omega$ is given by~\eqref{tCO}.

By~the above transformation rule for $h(z,-w)$, it~again also follows that
$$ f \longmapsto \frac{h(a,-a)^{-\nu/2}}{h(z,-a)^{-\nu}} f\circ\hphi^{-1},
 \qquad a=\hphi0, \; \hphi\in\hg, \; \nu\in\NN,  $$
is a projective unitary representation of $\hg$ on~$\hnmv$.
Let $\hat{\cS}^m$ be the vector space of all symmetric polynomials of degree $<m$ in $r$ variables, denote
$$ \hsmv := \hat{\cS}^m \cap L^2([0,1]^r,d\hrho_{b,\nu,a}),  $$
where
$$ d\hrho_{b,\nu,a}(t) := c_\Omega \; t^b (1-t)^\nu \prod_{1\le i<j\le r}|t_i-t_j|^a \, dt,  $$
and let $\hS^m_\nu(x,y)$ be the reproducing kernel of~$\hsmv$.
Proceeding as in Section~3 above, we~then obtain the following analogue of Theorem~\ref{PF}.

\begin{theorem} \label{QF}
\begin{itemize}
\item[(a)] For any $\hphi\in\hg$, the reproducing kernel $\hN^m_\nu$ of~$\hnmv$ satisfies
$$ \hN^m_\nu(z,w) = \frac{h(z,-a)^\nu h(a,-w)^\nu}{h(a,-a)^\nu} \hN^m_\nu(\hphi z,\hphi w), \qquad a:=\hphi^{-1}0 ;  $$
in~particular,
$$ \hN^m_\nu(z,w) = h(z,-w)^\nu \hN^m_\nu(\hphi_w z,0). $$
\item[(b)] Radial functions in $\hat{\cN}^m$ consist precisely of functions of the form
$$ p(z,(-z)^z), $$
where $p(z,w)$ is a polynomial in $z,\ow\in\CC^d$ of degree $<m$ in each argument
which is $K$-invariant in the sense of~\eqref{VE}.
\item[] Consequently, the radial functions in $\hat{\cN}^m$ coincide with the linear span of $K_\m(z,(-z)^z)$, $|\m|<m$.
\item[(c)] The mapping $\hat V$ from $\hat{\cS}^m$ into functions on $\hom$ given by
$$ \hat V f(k\sqrt x e) := f\Big(\frac x{1+x} \Big), \qquad k\in K, \; x\in\rpr,  $$
is a bijection from $\hat{\cS}^m$ onto radial functions in~$\hat{\cN}^m$.
Furthermore, $\hat V$ sends $\hsmv$ unitarily onto~the subspace $\hat{\cR}^m_\nu$
of radial function in~$\hnmv$, and
$$ \hN^m_\nu(\cdot,0) = V\hS^m_\nu(\cdot,0).  $$
\end{itemize}
\end{theorem}

\begin{proof} The~proof is the same as for Propositions~\ref{PB}, \ref{PD} and~\ref{PE}, hence omitted.
\end{proof}

Unlike the bounded case, for the compact dual we can give an explicit formula for the kernel
$\hN^m_\nu$ in terms of multivariable Jacobi polynomials $P^{(\alpha,\beta,a/2)}_\m$
(also called Heckman-Opdam polynomials; see~\cite[Section~1.3]{HS}).
Recall from \cite[Section~4.b]{Deb3} that $P^{(\alpha,\beta,a/2)}_\m(t)$ are symmetric polynomials
on $\RR^r$ such that
\begin{itemize}
\item[(i)] $P^{(\alpha,\beta,a/2)}_\m(t)$ is the symmetrization~of
\[  t^\m + \sum_{\n<\m} c_{\m\n} t^\n \label{JAC} \]
where the sum is over (some) signatures $\n$ smaller than $\m$ with respect to the lexicographic order; and
\item[(ii)] $P^{(\alpha,\beta,a/2)}_\m(t)$, $|\m|\ge0$, are orthogonal on~$[-1,+1]^r$ with respect
to the measure
$$ (1-t)^\alpha (1+t)^\beta \prod_{1\le i<j\le r}|t_i-t_j|^a \, dt . $$
\end{itemize}
By~change of variable, it~follows that $P^{(\alpha,\beta,a/2)}_\m(1-2t)$ are orthogonal on~$[0,1]^r$
with respect to the measure $t^\alpha(1-t)^\beta\prod_{1\le i<j\le r}|t_i-t_j|^a\,dt$,
with the norm-square on $[-1,+1]^r$ given by $2^{r(r-1)a/2+r\alpha+r\beta+r}$ times the norm-square on~$[0,1]^r$.
Setting in particular $\alpha=b$, $\beta=\nu$ we get an orthogonal basis for symmetric polynomials
with respect to $d\hrho_{b,\nu,a}$ on~$[0,1]^r$. By~\eqref{VA}, we~thus arrive at the following theorem.

\begin{theorem} For~$\nu\in\NN$, the reproducing kernel $\hS^m_\nu$ at the origin is given~by
\[ \hS^m_\nu(t,0) = \sum_{|\m|<m} \frac{2^{d+r\nu}P^{(b,\nu,a/2)}_\m(1)} {\|P^{(b,\nu,a/2)}_\m\|^2} \; P^{(b,\nu,a/2)}_\m(1-2t) . \label{SQ} \]
Here the norm-square is understood on $[-1,+1]^r$.
\end{theorem}

We~remark that explicit formulas both for $P^{(b,\nu,a/2)}_\m(1)$ and for $\|P^{(b,\nu,a/2)}_\m\|^2$
are available, see Theorems~3.5.5 and~3.6.6 in~\cite{HS}.

\begin{example} For~rank $r=1$, $P^{(b,\nu,a/2)}_\m$ are up to a constant factor just the ordinary
Jacobi polynomials $P^{(b,\nu)}_n$ of degree $n$ on~$[-1,+1]$:
$$ P^{(b,\nu,a/2)}_{(n)}(t) = \frac{2^n}{\binom{2n+b+\nu}n} P^{(b,\nu)}_n(t) .  $$
From the formulas \cite[Section~10.8]{BE2}
$$ P^{(b,\nu)}_n(1)=\binom{n+b}n, \qquad
\|P^{(b,\nu)}_n\|^2=\frac{2^{b+\nu+1}\Gamma(n+b+1)\Gamma(n+\nu+1)} {n!(2n+b+\nu+1)\Gamma(n+b+\nu+1)},  $$
we~therefore get
$$ \hS^m_\nu(t,0) = \sum_{j=0}^{m-1} \frac{\Gamma(j+\nu+1)} {(d-1)!j!^2(d+2j+\nu)\Gamma(d+j+\nu)} \; P^{(d-1,\nu)}_j(1-2t). \qed $$
\end{example}

Using Theorem~\ref{QF}, we~can also obtain from \eqref{SQ} a formula for the nearly-holomorphic
reproducing kernel $\hN^m_\nu(z,w)$ on~$\hom$ in terms of multivariable Jacobi polynomials.

\begin{corollary} \label{rep-ker-cpt}
For $\nu\in\NN$, the nearly-holomorphic reproducing kernel $\hN^m_\nu$ is given~by
$\hN^m_\nu(z,w)=h(z,-w)^\nu\hN^m_\nu(\hphi_w z,0)$, where
$$ \hN^m_\nu(k\sqrt xe,0) = \sum_{|\m|<m} \frac{2^{d+r\nu}P^{(b,\nu,a/2)}_\m(1)} {\|P^{(b,\nu,a/2)}_\m\|^2}
 \; P^{(b,\nu,a/2)}_\m \Big(\frac{1-x}{1+x}\Big) .  $$
\end{corollary}

\begin{proof} Straightforward from Theorem~\ref{QF} and~\eqref{SQ}.
\end{proof}

One~can also get the invariantly-polyanalytic kernels~$\hat P^m_\nu$.
We~leave the details (which are utterly routine) to~the interested reader.

\begin{remark}
Note that in \eqref{SQ} there is no restriction on $m_1$ in the sum, in~contrast to Corollary~\ref{COCO}
or~\eqref{WO}; the~reason being, of~course, that $d\hrho_{b,\nu,a}$ is a finite measure on~$[0,1]^r$
for $\nu\in\NN$, so~that the corresponding $L^2$ space contains all polynomials.
Still, proceeding as in Section~5, one~can get the following analogue of Conjecture~\ref{PK} for the compact case.
For $q,\nu\in\NN$, let~$\mathcal Q^q_\nu$ be the subspace of $L^2([0,1]^r,d\hrho_{b,\nu,a})$
spanned by $\{K_\m(te,e): \;m_1\le q\}$, and let $Q^q_\nu$ be its reproducing kernel.
Then it seems that
\[ Q^q_\nu(t,0) = \hat c^q_\nu \; \FK21{-q,\nu+p+q}pt ,  \label{GG}  \]
where
$$ \hat c^q_\nu = \frac {\gom(p+q)\gom(p+q+\nu)\gom(\qo)} {\pi^d \gom(p)\gom(q+\qo)\gom(q+\qo+\nu)}.  $$
This has been checked for the same set of values of $r,q,a,b$ as for Conjecture~\ref{PK}.

Note that in view of~\eqref{JACK} and~\eqref{JAC}, the Jacobi polynomials $P^{(b,\nu,a/2)}_\m(1-2t)$
with $m_1\le q$ form an orthogonal basis for~$\mathcal Q^q_\nu$, thus again by~\eqref{VA}
$$ Q^q_\nu(t,0) = \sum_{|\m|:\,m_1\le q} \frac{2^{d+r\nu}P^{(b,\nu,a/2)}_\m(1)} {\|P^{(b,\nu,a/2)}_\m\|^2} \; P^{(b,\nu,a/2)}_\m(1-2t) . $$
Hence \eqref{GG} gives a conjectured value for this sum.  \qed
\end{remark}

We~conclude this section by deriving the counterpart of Section~4, i.e.~the representation theory
of for the $L^2$ spaces of sections of line bundles --- especially the results of~\cite{GZkyoto}
--- for the compact~case. We~give a representation theoretic proof of Corollary \ref{rep-ker-cpt}.

We follow the presentation as in \cite{Lo}. We~consider the holomorphic line bundle $\mathcal L$ over $\hg/K$,
\[ \label{line-bd} \hg\times_{K, \tau} \CC\to \hom=\hg/K, \]
where $\tau(k) = (\det \text{Ad}(k)|_{\mathfrak p^+})^{1/p}$, $k\in K$. This is the holomorphic line bundle
such that $\mathcal L^p =\mathcal K^{-1}$ and it generates the Picard group of $\hom$; see \cite[7.1-7.11]{Lo}.
Here $\mathcal K^{-1}$ is the dual of the canonical line bundle.
Let $\tau_\nu =\tau^\nu$ for any fixed integer $\nu$, where as before we assume that $\nu\ge0$.

Let  $L^2(\hom; \nu)$ be the $L^2$-space of sections of the line bundle~$\mathcal L^\nu$. We normalize the measure
so that the realization of sections $f\in L^2(\hom; \nu)$ as functions on $L^2(\hg)$ is an isometry.
More precisely $L^2(\hom; \nu)$ consists of $ f\in L^2(\hg)$ such that
$$ \tau_\nu(k)f(gk) =f(g), \qquad k\in K, $$
and
$$ \Vert f\Vert_\nu^2=\int_{\hg} |f(g)|^2 \,dg<\infty , $$
where $dg$ is the Haar measure on $\hg$ normalized so that $\int_{\hg} dg=1$.

The space $V:=\CC^d$ can be realized as an open subset in $\hom$ and we shall realize the space $L^2(\hom; \nu)$
as point-wise functions on~$V$. Under our assumption $\nu\ge0$ the space of holomorphic sections of the line bundle
\eqref{line-bd} is non-zero, and there exists a global frame $e_\nu(z)$ with point-wise  norm
$$ \Vert e_\nu(z)\Vert_z^2 = h(z, -z)^{-\nu}. $$
Then a section $f\in L^2(\hom; \nu)$ will be written as $f= f(z) e_\nu(z)$ for a point-wise function on $V$ such that
$$  \Vert f\Vert^2_{\nu} = \hat c_\nu \int_{V} |f(z)|^2 h(z, -z)^{-\nu} d\mu_0(z), \qquad  f= f(z) e_\nu(z), $$
where
$$ d\mu_0(z) =\frac{dz}{h(z, -z)^{p}} $$
is the $\hg$-invariant (K\"a{}hler) measure on~$\hom$. To avoid confusion we write $L^2(V, \nu)$ for the space
of $L^2$-functions $f(z)$ with the above norm. As~an $L^2$-space and unitary representation of $\hg$,
$L^2(\hom; \nu)=L^2(V, \nu)$ via this identification.

Let $\g^{\CC}=\mathfrak p^{+} + \mathfrak k^{\CC} + \mathfrak p^{-}$ be the  Harish-Chandra decomposition of $\g^{\CC}$.
We~use the same complex structure for $\Omega$ as for~$\hom$, so that $\mathfrak p^{+}=T_0^{(1, 0)}(\hom)\equiv V$
is the holomorphic tangent space at $0\in V\subset \hom$. Let $\mathfrak t\subset \mathfrak k^{\CC}$
be a Cartan subalgebra, and let $\gamma_1> \cdots >\gamma_r$ be the Harish-Chandra strongly ortogonal roots so
that $\gamma_1$ is the highest root for $\mathfrak p^+$ as representation of $\mathfrak k^{\CC}$.
In particular $\gamma_1$ is the highest root of $\g^{\CC}$ as representation of $\g^{\CC}$.
Let $\mathfrak t^-$ be the span of the co-roots of $\gamma_1, \cdots, \gamma_r$ and let
$\mathfrak t =\mathfrak t^{-} + \mathfrak t^{+}$ with $\gamma_1, \cdots, \gamma_r$ vanishing on $\mathfrak t^{+}$.
The root space decomposition of $\g^{\CC}=\mathfrak p^{+} + \mathfrak k^{\CC} + \mathfrak p^{-}$ is refined as
$\g^{\CC}=(\mathfrak p^{+} + \mathfrak k^{+}) + \mathfrak t + (\mathfrak k^{-} + \mathfrak p^{-})$
with $\mathfrak k^{-} + \mathfrak p^{-}$ the space of negative root vectors,
and  ${\mathfrak k^+ +\mathfrak k^-}\subset [\mathfrak k^{\CC},\mathfrak k^{\CC}]$.

The $L^2$-space $L^2(\hom, \nu)$ is decomposed as
\[   \label{eq:L-2-dec} L^2(\hom, \nu)=\sum_{\m} V_{\nu, \m} \]
where  each $V_{\nu, \m}$ is of highest weight whose restriction on $\mathfrak t^-$~is
$$ \frac\nu2 +m_1 \gamma_1 +\cdots + m_r\gamma_r , \qquad \frac\nu2:=\frac12 \nu(\gamma_1 +\cdots + \gamma_r),  $$
where $m_j$ are nonnegative integers subject to the condition
\[\label{Sch}   m_1\ge \cdots \ge m_r \ge  0. \]
When  $\Omega=G/K$ is not of tube type this does not define completely the highest weights and it requires
some extra specifications; however the highest weights of these representations that appear in $L^2(\hom, \nu)$
are uniquely determined by the condition above, see \cite{Sci}, \cite{Shm},~\cite{GZcomp}.

Recall the $\tau_\nu$-spherical  functions on $\hg$
$$ \tau(k_1)^\nu \tau(k_2)^\nu f(k_1 gk_2)= f(g), \qquad g\in \hg, \; k_1, k_2\in K .  $$
As functions on $\hg$ each space $V_{\nu, \m}$ contains a unique $\tau_\nu$-spherical function $\Psi_{\nu, \m}$
normalized by $\Psi_{\nu,\m}(e)=1$. We~set
$$ \phi_{\nu, \m}(z) = J_g(0)^{-\frac \nu{p}} \Psi_{\nu, \m}(g),  \qquad g\cdot 0= z ,$$
as a trivialization of the $\tau_\nu$-spherical function $\Psi_{\nu, \m}$.
In particular $\phi_{\nu, \m}(z)$ is now both left and $K$-invariant, and thus can be realized as a left $K$-invariant
function on $V\subset \hom$, $\phi_{\nu, \m}(kz)=\phi_{\nu, \m}(z)$, $\phi_{\nu, \m}(0)=1$, and
$$ \phi_{\nu, \m}(z) = h(z, -z)^{-\frac{\nu} 2} \Psi_{\nu, \m}(\hphi_z). $$
In the notation above $\phi_{\m}(z)$ is the coefficient of the section $\Psi_{\nu, \m}$ with respect to
the frame $e_\nu(z)$. The orthogonality relations for $\phi_{\nu, \m}$ read~now
\begin{multline*}
\hat c_{\nu} \int_{V}\phi_{\nu, \m}(z) \overline{\phi_{\nu, \m'}(z)} h(z, -z)^{-\nu} d\mu_0(z)  \\
= \hat c_\nu c_\Omega 2^r \int_{\rpr}\phi_{\nu, \m}(x) \overline{\phi_{\nu, \m'}(x)} \prod_{j=1}^r (1+x_j^2)^{\nu -p}
  \prod_{1\le j<k\le r} (x_j^2-x_k^2)^{a} \prod_{j=1}^r x_j^{2b+1} \, dx_j   \\
= \frac{1}{d_{\nu, \m}} \delta_{\m, \m'} ,
\end{multline*}
where $d_{\nu, \m} =\dim V_{\nu, \m}$ is the dimension of $V_{\nu, \m}$
(which can be computed using the Weyl dimension formula).
These are the Jacobi polynomials of Heckman and Opdam. 
(The functions $\Psi_{\nu, \m}$ are the spherical functions $\phi_{\lambda, \nu}$ studied by Shimeno
for specific discrete values of the parameter~$\lambda$; see~\cite[Remark 5.12]{Shm}.)

In~particular, for $\m=(0)$, $V_{\nu,(0)}$ is the Bergman space of holomorphic sections
of the line bundle defined by $\nu$ in $L^2(\hom, \nu)$. It~can be realized as the space
of holomorphic polynomials of degree $\le \nu$ and has reproducing kernel $\hat c_\nu h(z, -w)^{\nu}$.
The corresponding Heckman-Opdam polynomial is the constant function  $\phi_{\nu,(0)}(z)=1$.

We~equip $\hom$ with the $\hg$-invariant (K\"a{}hler) metric and let $\overline D$ be the
associated invariant Cauchy-Riemann operator. We describe the decomposition \eqref{eq:L-2-dec}
using the kernels of $\overline D^m$. We~shall need some results on the vanishing properties
of Shimura operators on the spaces $V_{\nu, \m}$  obtained in \cite{SaZ}. First we recall
the Shimura operators using our present formulation. Recall from Section~2 the Hua-Schmid decomposition
$$ \otimes^m V =\sum_{|\m|=m} S^{\m} V $$
of the symmetric tensor product $\otimes^m V$ under~$K$. Let $P_{\m}$ be the corresponding projection.
It~is a general fact that $\bar D^m: C^\infty(G, K; \tau_\nu) \to C^\infty(G, K; \tau_\nu\otimes \otimes^m V)$,
where as before $V$ is identified as the holomorphic tangent space $T_0^{(1, 0)}(\hom)$ of $\hom$ at~$0$,
and $C^\infty(G, K; \tau_\nu\otimes \otimes^m V)$ is the space of smooth sections of the line bundle
$\mathcal L^\nu\otimes\otimes^m T^{(1, 0)}$ realized as functions on $\hg$ transforming under $K$ as
$$ \tau_\nu(k)\otimes^m \text{Ad}(k) f(gk) = f(g), \qquad g\in \hg.  $$
The Shimura operators are defined by
$$ L_{\m} =(\bar D^{|\m|})^\ast P_{\m} \bar D^{|\m|}.   $$
We~have then
$$ (\bar D^{m+1})^\ast \bar D^{m+1}=\sum_{|\m|=m+1} L_{\m}.  $$

\begin{theorem} \label{abs-ker}
The kernel $\operatorname{Ker} \bar D^{m+1}$ in $L^2(\hom, \nu)$ is precisely the direct sum
$$ \operatorname{Ker} \bar D^{m+1} =\sum_{|\m| \le m}^{\oplus} V_{\nu, \m}.   $$
In~particular the reproducing kernel at the origin for the space of of nearly holomorphic sections
of order $m+1$ in $L^2(\hom, \nu)$ is given~by
$$ \hN_\nu^{m+1}(z, 0)=  \sum_{|\m|\le m} d_{\m}\phi_{\nu, \m} (z). $$
\end{theorem}

\begin{proof} The operator $L_{\m}$ acts on each irreducible component $ V_{\nu, \n}$ in \eqref{eq:L-2-dec}
as a non-negative scalar multiple of the identity, by~Schur's lemma, and their eigenvalues are shown in
\cite{SaZ} to be given by Okounkov polynomials. More precisely, the~eigenvalue of $L_{\m}$ on $V_{\nu,\n}$
is a symmetric polynomial $\widetilde L_{\m}(\frac{\nu} 2 +\n +\rho )$ of $\frac{\nu} 2 +\n +\rho$, where
$\rho$ is the half-sum of positive roots of $\mathfrak t$ in $\mathfrak g^{\CC}$.
(One~may also take $\frac{\nu}2$ into the definition of $\rho$ as above.)
It~follows from \cite[Theorem 5.1]{SaZ} that $\widetilde L_{\m}(\frac{\nu} 2 +\n +\rho )=0$
unless $\m\subseteq \n$ (i.e.~$m_j\le n_j$ for all $j=1,\dots,r$). This implies that
\[ \sum_{|\n| \le m}^{\oplus} V_{\nu, \n} \subseteq \operatorname{Ker} (\bar D^{m+1})^\ast \bar D^{m+1} = \operatorname{Ker} \bar D^{m+1}.   \label{incl} \]

Now we prove the reverse inclusion, namely that if $|\n| >m $ then $\bar D^{m+1}$ on $V_{\nu, \n}$ is non-zero.
Suppose to the contrary that $\bar D^{m+1}:V_{\nu, \n}\to 0$. We use  the formulation as in \cite[Section 3.4]{SaZ}
for the realization of $V_{\nu, \n}$ to compute the action of $\bar D^{m+1}$.
As~a unitary representation $(V_{\nu,\n}, \hg, \pi_{\n})$ of $\hg$, the space $V_{\nu, \n}$ contains
a unique non-zero vector $v_{\nu}$ such that
$$ \pi_{\n}(k) v_{\nu} =\tau_\nu(k) v_{\nu}   $$
where $\tau_\nu$ is the one-dimensional representation defined as above. Moreover both representations
$\tau_{-\nu}$ and $\tau_{\nu}$ appear in~$V_{\nu, \n}$. As~functions on~$G$,
the space $V_{\nu, \n}\subset L^2(\hom, \nu)\subset L^2(\hg)$ is obtained~as
$$ v\in V_{\nu, \n} \mapsto f_v(g)=\langle \pi_{\n}(g^{-1}) v, v_{-\nu}\rangle,
 \qquad f_v\in V_{\nu, \n}\subset L^2(\hg),  $$
where with some abuse of notation we have used the same notation $V_{\nu, \n}$ both as $\hg$-representation and
as a space of functions. The assumption $\bar D^{m+1}: V_{\nu, \n}\to 0$ implies in particular that
$\bar D^{m+1}f_{v_{-\nu}}=0$, and its evaluation at $g=e$ implies further that
$$ \pi_{\n}(X) v_{-\nu}=0 $$
for all $X\in S^{m+1}(\mathfrak p^-)$.
Let $X =X_1 Y$ where $X_1\in \mathfrak p^-$ is an arbitrary negative root vector and $Y\in S^{\m}(\mathfrak p^-)$
is a {\it $\mathfrak k^{+}$-lowest weight vector} in $S^{\m}(\mathfrak p^-)$ with lowest weight
$-(m_1 \gamma_1 +\cdots +m_r\gamma_r)$ with $m_1\ge \cdots\ge m_r\ge0$.
(A~construction of all lowest weight vectors is found in \cite{Up86} but we shall not need the explicit form.)
We~have then $\pi_{\n}(X_1)\pi_{\n}(Y) v_{-\nu}=0$. Since $v_{-\nu}$ defines a one-dimensional representation of
$\mathfrak k^{\mathbb C}$ we have always $\pi_{\n}(X)\pi_{\n}(Y) v_{-\nu}=0$, for $X\in \mathfrak k^{-}
\subset [\mathfrak k^{\CC}, \mathfrak k^{\CC}]$.
In~other words, $\pi_{\n}(Y) v_{-\nu}$ is a lowest weight vector for the $\mathfrak g^{\mathbb C}$-representation
unless it vanishes. However by the Hua-Schmid decomposition the element $\pi_{\n}(Y) v_{-\nu}$ has lowest weight
$-\frac{\nu}2 -(m_1 \gamma_1 +\cdots +m_r\gamma_r)$, $m_1+\cdots + m_r=m <|\n|$.
But~the space $V_{\nu, \n}$ has lowest weight $-\frac{\nu}2 -(n_1 \gamma_1 +\cdots +n_r\gamma_r)$
and thus $\pi_{\n}(Y) v_{-\nu}=0$.
Acting by $k\in K$ we find $\pi_\n({\text{Ad}(k)Y}) \pi_\n(k)v_{-\nu}=0$.
Again $v_{-\nu}$ defines a one-dimensional representation of $K$, $\pi_\n(k) v_{-\nu}=\tau_{-\nu}(k)v_{-\nu}$
with scalar character $\tau_{-\nu}(k)$. Thus $\pi_\n({\text{Ad}(k)Y})v_{-\nu}=0$.
Furthermore $\{\text{Ad}(k)Y, k\in K\}$ generates the irreducible the representation $S^{\m}(\mathfrak p^-)$
so we get $\pi_\n(X)v_{-\nu}=0$ for all $X \in S^{\m}(\mathfrak p^-)$,
and further $\pi_\n(X)v_{-\nu}=0$ for all $X \in S^{m}(\mathfrak p^-)$.
Continuing this procedure we get that $v_{-\nu}=0$, a~contradiction.
This proves our claim on $\operatorname{Ker}\bar D^{m+1}$ and then on the reproducing kernel.
\end{proof}

\begin{remark} As~noted in  \cite[Remark~5.12]{Shm} the spherical  functions  $\phi_{\nu, \m} (z)$ here
are precisely the Heckman-Opdam polynomials in Corollary~\ref{rep-ker-cpt} under proper coordinate change.
Thus Theorem~\ref{abs-ker} is just an abstract restatement and a different proof of the expansion in
Corollary~\ref{rep-ker-cpt} (with $m+1$ replacing $m$ for notational convenience) with interpretation
of the coefficients using the dimension~$d_{\nu, m}$.  \qed
\end{remark}

\begin{remark} The subspace $V_{\nu, \m}$ can also be described using, as~in Section~3, the quasi-inverse
$\partial \log h(z, -z)$. In~the local coordinates $z\in V\subset \hom$ the space $V_{\nu, \m}$ consists
of functions
$$ f(z) = \otimes^m(\partial \log h(z, -z)) (F(z))  $$
where $F$ is a holomorphic section of the bundle $\mathcal L^\nu \otimes \otimes^m T^{(1, 0)}$ in the highest
weight representation above.   \qed
\end{remark}

\begin{remark} It~follows from the proof above that for any  $\n$ there exists an $\n'$,
$|\n'|\le |\n|$ such that the eigenvalue  $\widetilde L_{\n'}(\frac{\nu} 2 +\n + \rho )$
of the Shimura operator on $V_{\nu, \n}$ is nonvanishing, $\widetilde L_{\n'}(\frac{\nu} 2 +\n +\rho ) \neq 0$.
This might be a known fact or can be proved by using Koornwinder's formula
(see~\cite{Koor}, \cite{Ok} and \cite[Theorem 5.5]{SaZ}) for $\widetilde L_{\n'}$,
which in turn can give a different proof of the reverse inclusion of~\eqref{incl}.  \qed
\end{remark}

\begin{example} Let~us again make everything more specific for the rank one case,
i.e.~when $\hg/K=\CC P^d$ is the complex projective space.
In~this case it is more convenient to use the realization of $\CC P^d$ as $\CC P^d=U(d+1)/U(d)\times U(1)$.
We~choose the Cartan subalgebra of $U(d+1)$ as diagonal matrices identified as $\RR^{d+1}$,
with the Harish-Chandra root $\beta =(1, 0, \cdots, 0, -1)$.
The~highest weights above are now $(\nu + m, 0, \cdots, -m)$.
The~sections of the line bundle with parameter $\nu$ on $\CC P^d=U(d+1)/U(d)\times U(1)$
can be realized as functions on the sphere $S^{2d+1}=U(d+1)/U(d)$
and the representation space with the highest weight $(\nu + m, 0, \cdots, -m)$ is the
space of $(p, q)=(\nu +m, m)$-spherical harmonic polynomials. We write $\phi_{\nu,(m)}=\phi_{\nu,m}$.

When  $\nu=0$, i.e.~the spherical case, the highest weight is of the form $\m= m\beta$ with spherical polynomial
$$\phi_{0,m}(\exp(H)) = \F21{-m, d+m}d{\sin^2 \frac{\beta(H)}2};  $$
see \cite[Theorem V.4.5]{He} and Remark~\ref{HELGASON} above. For~general $\nu\ge 0$,
$$ \phi_{\nu, m}(\exp(tH)) = \F21{d+m+\nu, -m}d{\sin^2 t}. $$
See~\cite{SaZ},~\cite{JW}.

By~the Schur orthogonality we have
$$ \langle \phi_{\nu, m},\phi_{\nu, m'}\rangle = \frac{1}{d_{\nu, m}} \delta_{m, m'},  $$
where $d_{\nu, m}$ is the dimension of the representation space $V_{\nu, m}$.
Here the~inner product is given~by
\begin{multline*}
 \langle\phi, \psi\rangle = \hat c_\nu \int_{0}^{\frac{\pi}2} \phi(\sin^2 t) \overline{\psi(\sin^2 t)} \sin^{2\nu +1}(2t) \sin^{2(d-1) -2\nu}(t) dt \\
= \hat c_\nu \int_{0}^{\frac{\pi}2} \phi(\sin^2 t) \overline{\psi(\sin^2 t)} \sin^{2\nu }(2t) \sin^{2(d-1) -2\nu}(t) d\sin^2 t  \\
= \hat c_\nu \int_{0}^{1} \phi(x) \overline{\psi(x)} (1-x)^{\nu }x^{d-1} dx.
\end{multline*}
The $\tau_\nu$-spherical function above is
$$ \phi_{\nu, m}(x) = \F21{-m, m+d +\nu}dx. $$
The dimension of the representation space $V_{\nu, m}$ can be easily found using the Weyl dimension formula and equals
$$ d_{\nu, m} =\frac{(2m+\nu +d) (m+\nu+1)_{d-1} (m+1)_{d-1}}{ d! (d-1)!}.  $$
In~particular,
$$ d_{\nu,0} =\frac{(\nu +d) (\nu+1)_{d-1} }{ d!} =\frac{(\nu+1)_{d} }{ d!} =\binom{\nu +d}{d}  $$
which is precisely the dimension of the space of polynomials $\mathcal  P_{\le \nu}(\CC^d)$
on $\CC^d$ of degree $\le \nu$ realized as the holomorphic sections in $L^2(\hom, \nu)$.

So~we are computing the sum
\begin{multline*}
 \sum_{m \le n} d_{\nu, m} \phi_{\nu, m}(x) = \\
 \sum_{m \le n} \frac{(2m+\nu +d) (m+\nu+1)_{d-1} (m+1)_{d-1}}{ d! (d-1)!}  \; \F21{-m, m+d +\nu}dx.
\end{multline*}
To carry out the summation we use the following elementary observation.

\begin{lemma}  \label{lem-Y}
Let $d\mu(x)$ be a finite Borel measure on $\RR_+$ such that all polynomials are dense in $L^2(\RR_+, d\mu)$.
Let $\{p_m\}_{m=0}^\infty$ be the orthonormal basis obtained from the Gram-Schmidt orthogonalization
of the polynomials $\{x^m\}_{m=0}^\infty$. Then the reproducing kernel $\sum_{m=0}^n p_m(x) p_m(0)$
evaluated at $0$ is
$$ \sum_{m=0}^n p_m(x) p_m(0) = A_n q_n(x)   $$
for some constant $A_n$, where $\{q_n(x)\}_{n=0}^\infty$ is the orthonormal basis obtained from
$\{x^n\}_{n=0}^\infty$ for the space $L^2(\RR_+, d\tilde \mu)$, where $d\tilde \mu=x\,d\mu(x)$.
\end{lemma}

\begin{proof} Write $P_n(x)= \sum_{m=0}^n p_m(x) p_m(0)$. We~prove that $P_n(x)$ is orthogonal
to all polynomials $x^m$, $0\le m\le  n-1$, in~the space $L^2(\RR_+, d\tilde \mu)$.
Indeed the inner product of $x^n$ and $P_m$ in $L^2(\RR_+, d\tilde \mu)$~is
$$ \int_0^\infty x^m P_n(x) x \, d\mu(x)
 = \int_0^\infty x^{m+1} P_n(x) \, d\mu(x) =x^{m+1}|_{x=0}=0 , $$
since $P_n(x)$ is the reproducing kernel at 0 in $L^2(\RR_+, d\mu)$ for the polynomials of degree
$\le n$ and $0<m+1\le n$. Thus $P_n$ is proportional to~$q_n$. This proves the lemma.
\end{proof}

\begin{theorem} The reproducing kernel $\hN^n_\nu(z,0)$ at the origin for the space $\hN^n_\nu(\CC P^d)$,
under the local trivialization above using the local frame $e_\nu$ on $\CC^d\subset \CC P^d$,~is
$$ \hN^n_\nu(z, 0)= \sum_{m \le n} d_{\nu, m} \phi_{\nu, m}(x) = A_n \; \F21{-n, n+d +\nu +2}{d+1}x,
 \quad x=\frac{|z|^2}{{1+|z|^2}},  $$
where the positive constant $A_n$ is given by \eqref{AN} below.
\end{theorem}

\begin{proof} We~use Lemma~\ref{lem-Y}. The polynomials $\{\phi_{\nu, m}(x)\}$ form an orthogonal basis
for the space $L^2((0, 1), d\mu(x))$, $d\mu(x)= (1-x)^{ \nu }x^{ d-1}\,dx$,
and they are the same orthogonal basis as obtained from the Gram-Schmidt process from the measure~$d\mu(x)$.
The~orthogonal basis for the measure $d\tilde\mu(x)= x\,d\mu(x)=(1-x)^{\nu } x^{d+1}$ is
$\F21{-m, m+d +\nu +2}{d+1}x$. Thus
\[ \begin{split} \label{A-n}
 \sum_{m \le n} d_{\nu, m} \phi_{\nu, m}(x) & = \sum_{m \le n} d_{\nu, m} \phi_{\nu, m}(x)  \\
&= A_n \; \F21{-n, n+d +\nu +2}{d+1}x
\end{split}  \]
for some constant~$A_n$. To~find $A_n$, we~view (\ref{A-n}) as an identity of two polynomials of $x\in\RR$.
The~leading coefficients of $x^n$ in (\ref{A-n})~are
$$ d_{n, \nu}\frac{(-n)_n (n+d+\nu)_n}{(d)_n n!} = A_n \frac{(-n)_n (n+d+\nu+2)_n}{(d+1)_n n!}. $$
Thus
\[ \label{AN} \begin{split}
A_n&= d_{n, \nu} \frac{(n+d+\nu)_n  (d+1)_n }{(n+d+\nu+2)_n  (d)_n}
 = d_{n, \nu} \frac{(n+d+\nu)_n  (d+n) }{(n+d+\nu+2)_n  d}  \\
&= \frac{(2n+\nu +d) (n+\nu+1)_{d-1} (n+1)_{d-1}} { d! (d-1)!}\frac{(n+d+\nu)_n  (d+n) }{(n+d+\nu+2)_n  d}  \\
&=\frac{(n+\nu +1)_{d+1} (n+1)_{n+d-1}} {(2n+d+\nu +1) d!^2}.
\end{split}  \]
\end{proof}

%
\end{example}


\begin{thebibliography}{DOZ}

\bibitem[Ank]{Ank} J.-P. Anker: {\it An introduction to Dunkl theory and its analytic aspects,\/}
Analytic, algebraic and geometric aspects of differential equations
(G.~Filipuk, Y.~Haraoka, S.~Michalik,~eds.), Trends in Math., Birkh\"auser, 2017, pp.~3--58.

\bibitem[Ar]{Asurv} J. Arazy: {\it A~survey of invariant Hilbert spaces of analytic
functions on bounded symmetric domains,\/} Multivariable operator theory
(R.E.~Curto, R.G.~Douglas, J.D.~Pincus, N.~Salinas,~eds.), Contemp. Math.~185,
AMS, Providence, 1995, pp.~7--65.

\bibitem[AE]{E61} J. Arazy, M. Engli\v s: {\it $Q_p$-spaces on bounded symmetric domains,\/}
J. Funct. Spaces Appl. {\bf 6} (2008), 205--240.

\bibitem[BE]{BE} H. Bateman, A. Erd\'elyi, {\it Higher transcendental functions,
vol.~I,\/} McGraw-Hill, New York - Toronto - London 1953.

\bibitem[BE2]{BE2} H. Bateman, A. Erd\'elyi, {\it Higher transcendental functions,
vol.~II,\/} McGraw-Hill, New York - Toronto - London 1953.

\bibitem[BO]{BO} R.J. Beerends, E.M. Opdam: {\it Certain hypergeometric series related to the root system BC,\/}
Trans. Amer. Math. Soc. {\bf 339} (1993), 581--609.

\bibitem[BM]{BM} S. Bochner, W.T. Martin, {\it Several complex variables,\/}
Princeton University Press, Princeton, 1948.

\bibitem[De1]{DebSh3} A. Debiard: {\it Syst\`eme diff\'erentiel hyperg\'eom\'etrique et parties radiales
des op\'erateurs invariants des espaces sym\'etriques de type $BC_p$,\/} Lecture Notes Math.~1296,
Springer, New~York-Berlin, 1988, pp.~42--124.

\bibitem[De2]{DebSh4} A. Debiard: {\it Le syst\`eme diff\'erentiel hyperg\'eom\'etrique de type BC
et son spectre,\/} C.~R.~Acad. Sci. Paris Ser.~I Math. {\bf 314} (1992), 793--796.

\bibitem[De3]{Deb3} A. Debiard: {\it Op\'erateurs diff\'erentiels d'Euler-Poisson-Darboux g\'n\'eralis\'es
et fonctions hyperg\'eom\'etriques de plusieurs variables,\/} Complex analysis, harmonic analysis and applications,
Pitman Res. Notes Math. Ser.~347, Longman, Harlow, 1996, pp.~56--73.

\bibitem[DOZ]{DOZ} A.H. Dooley, B. \O rsted, G. Zhang: {\it Relative discrete series of line bundles
over bounded symmetric domains,\/} Ann. Inst. Fourier {\bf 46} (1996), 1011--1026.

\bibitem[En]{E25} M. Engli\v s: {\it The asymptotics of a Laplace integral on a K\"ahler manifold,\/}
J.~reine angew. Math. {\bf 528} (2000), 1-39.

\bibitem[EZ]{EZcr} M. Engli\v s, G. Zhang: {\it Toeplitz operators on higher Cauchy-Riemann spaces,\/}
Documenta Math. {\bf 22} (2017), 1081--1116.

\bibitem[FK]{FK88} J. Faraut, A. Kor\'anyi: {\it Function spaces and reproducing
kernels on bounded symmetric domains,\/} J.~Funct. Anal. {\bf 88} (1990), 64--89.

\bibitem[HS]{HS} G. Heckman, H. Schlichtkrull, {\it Harmonic analysis and special functions on symmetric spaces,\/}
Academic Press, 1994.

\bibitem[He]{He} S. Helgason, {\it Groups and geometric analysis,\/} Academic Press, Orlando, 1984.

\bibitem[JW]{JW} K.~D.~Johnson, N.~R.~Wallach: {\it Composition series and intertwining operators
for the spherical principal series.~I,\/} Trans. Amer. Math. Soc. {\bf 229} (1977), 137--173.

\bibitem[Ka]{Ka} J. Kaneko: {\it Selberg integrals and hypergeometric functions associated with Jack polynomials,\/}
SIAM J. Math. Anal. {\bf 24} (1993), 1086--1110.

\bibitem[Ko]{Koor} T. Koornwinder: {\it Okounkov's BC-type interpolation Macdonald polynomials and their $q=1$ limit,\/}
S\'em. Lothar. Combin. {\bf B72a} (2015), 27~pp.

\bibitem[Lo]{Lo} O. Loos, {\it Bounded symmetric domains and Jordan pairs,\/}
University of California, Irvine, 1977.

\bibitem[MD]{MD} I.G. MacDonald, {\it Symmetric functions and Hall polynomials\/}
(2nd~edition), Clarendon Press, Oxford, 1995.

\bibitem[Ok]{Ok} A. Okounkov: {\it $BC$-type interpolation Macdonald polynomials and binomial formula for Koornwinder polynomials,\/}
Transform. Groups {\bf3} (1998), 181--207.

\bibitem[Op]{Opd} E.M. Opdam: {\it Harmonic analysis for certain representations of graded Hecke algebras,\/}
Acta Math. {\bf 175} (1995), 75--121.

\bibitem[PZ]{PZ} J. Peetre, G. Zhang: {\it Invariant Cauchy-Riemann operators and relative discrete series
of line bundles over the unit ball of~$\CC^d$,\/} Michigan Math. J. {\bf 45} (1998), 387--397.

\bibitem[SaZ]{SaZ} S. Sahi, G. Zhang: {\it Positivity of Shimura operators,\/} Math. Res. Lett. {\bf 26} (2019), 587--626.

\bibitem[Sci]{Sci} H.~Schlichtkrull: {\it One-dimensional K-types in finite dimensional representations of
semisimple Lie groups: A generalization of Helgason's theorem,\/} Math. Scand. {\bf 54} (1984), 279--294.

\bibitem[Sch]{Sch} W. Schmid: {\it Die Randwerte holomorpher Funktionen auf
hermitischen R\"aumen,\/} Invent. Math. {\bf 9} (1969), 61--80.

\bibitem[Shm]{Shm} N. Shimeno: {\it The Plancherel formula for spherical functions with a one-dimensional $K$-type
on a simply connected simple Lie group of Hermitian type,\/} J. Funct. Anal. {\bf 121} (1994), 330--388.

\bibitem[Shi]{Shi} G. Shimura, {\it On a class of nearly holomorphic automorphic forms,\/}
Ann. Math. {\bf 123} (1986), 347-406.

\bibitem[Up1]{Up86} H. Upmeier, {\it Jordan algebras and harmonic analysis on symmetric spaces,\/}
Amer. J. Math. {\bf 108} (1986), 1--25.

\bibitem[Up2]{Up} H. Upmeier, {\it Toeplitz operators and index theory in several
complex variables,\/} Operator Theory: Advances and Applications~81,
Birkh\"auser Verlag, Basel, 1996.

\bibitem[Yan]{Yan} Z. Yan: {\it A class of generalized hypergeometric functions in
several variables,\/} Canad. J. Math. {\bf 44} (1992), 1317--1338.

\bibitem[You]{You} E.-H. Youssfi: {\it Polyanalytic reproducing kernels in~$\CC^n$,\/}
Complex Anal. Synergies {\bf 7} (2021), art.~no.~28.

\bibitem[Zh1]{GZstud} G. Zhang: {\it A~weighted Plancherel formula~II. The~case of the ball,\/}
Studia Math. {\bf 102} (1992), 103--120.

\bibitem[Zh2]{GZcomp} G.~Zhang: {\it Berezin transform on compact Hermitian symmetric spaces,\/}
Manuscripta Math. {\bf 97} (1998), 371--388.

\bibitem[Zh3]{GZkyoto} G. Zhang: {\it Nearly holomorphic functions and relative discrete series
of weighted $L^2$-spaces on bounded symmetric domains,\/} J.~Math. Kyoto Univ. {\bf 42} (2002), 207--221.

\end{thebibliography}
\end{document}